\documentclass[11pt]{article}
\usepackage{amsthm}
\usepackage{amsmath}  
\usepackage{amssymb}  
\usepackage{geometry}  
\usepackage{hyperref}  
\usepackage{authblk}  
\usepackage{setspace}  
\usepackage{pxfonts}
\usepackage{yfonts}
\usepackage{dsfont}
\usepackage{marvosym}
\usepackage{graphicx}
\usepackage{relsize}
\usepackage[latin1]{inputenc}
\usepackage{enumerate}

\def\bel{\begin{equation}\label}
\def\eeq{\end{equation}}
\def\ds{\displaystyle}

\def\R{\mathbb R}
\def\Z{\mathbb Z}

\def\F{\mathcal{F}}

\def\D{{\bf D}}

\def\B{{\bf B}}

\def\Tilde{\widetilde}

\def\bar{\overline}
\def\supp{{\bf supp}}

\def\alpha{\alphaup}
\def\beta{\betaup}
\def\gamma{\gammaup}
\def\delta{\deltaup}
\def\xi{{\xiup}}
\def\eta{{\etaup}}
\def\tau{{\tauup}}
\def\rho{{\rhoup}}
\def\phi{{\phiup}}
\def\psi{{\psiup}}
\def\lambda{{\lambdaup}}
\def\omega{\omegaup}
\def\varphi{{\varphiup}}
\def\gamma{{\gammaup}}

\def\var{\varepsilon}

\newtheorem{theorem*}{Theorem}
\newtheorem{thm}{Theorem}[section]

\newtheorem{lemma}{Lemma}[section]

\newtheorem{defn}{Definition}[section]

\date{}
\title{ \bf Homogeneous~Besov~spaces~in~Dunkl~setting}
\author[1]{Mengmeng Dou}
\author[2]{Jiashu Zhang}
\affil[1]{School of Mathematical Sciences, Zhejiang University,  Hangzhou, 310027,  People's Republic of China}
\affil[2]{School of Mathematical Sciences, Fudan University, Shanghai, 200433, People's Republic of China}
\begin{document}
\maketitle
\renewcommand{\thefootnote}{\fnsymbol{footnote}}
\footnotetext[1]{Emails of corresponding authors: doumengmeng@westlake.edu.cn;
	
~~~~~~~~~~~~~~~~~~~~~~~~~~~~~~~~~~~~~~~~~~~~~~~~~~~~~~~~~~~~~~~~	zhangjiashu@westlake.edu.cn.}
\footnotetext[2]{the ORCID of the first author: https://orcid.org/0009-0000-2022-8369.}
\begin{abstract}
{The purpose of this paper is to characterize the homogeneous Besov spaces in the Dunkl setting. We utilize a new discrete Calder\'{o}n reproducing formula, that is, the building blocks are differences of the Dunkl-Poisson kernel which involves both the Euclidean metric and the Dunkl metric. To introduce  the Besov spaces in the Dunkl setting, new test functions  and distributions are introduced. Also we prove a weak-type discrete Calder\'{o}n reproducing formula for the Dunkl-Besov spaces and demonstrate the completeness of the Dunkl-Besov spaces. As an application, we study the boundedness of Dunkl-Calder\'{o}n-Zygmund operators on the Dunkl-Besov spaces. }

\end{abstract}

	\section{Introduction}
	\setcounter{equation}{0}	
	It is well known that harmonic analysis is strongly influenced by the group structures. An important example is the Dunkl setting, which is associated with a reflection group on $\R^n$. The  purpose of this paper is to establish the Besov spaces in the Dunkl setting. We begin by recalling the  homogeneous  Besov spaces on the usual Euclidean spaces.
	%We begin with recalling the Frazier and Jawerth's decompositions of the classical Besov spaces.
	
	On the Euclidean spaces, for $-\infty<s<\infty$ and $0<p,q\leq\infty$, the homogeneous Besov space $\dot{B}^s_{p,q}$ is defined to be the collection of all distributions $f\in \mathcal{S}'/{\mathcal P}$ (tempered distribution modulo polynomials) such that (see \cite{Trible})
	\begin{equation}
		\|f\|_{\dot{B}_{p,q}^s}=\left\{\sum_{j\in \Z}\left\|2^{-js}\varphi_j*f\right\|_p^q \right\}^{1\over q}<\infty
	\end{equation}
	with usual interpretation $p,q=\infty$,	where $\varphi_j\in \mathcal{S}$ (Schwartz function) satisfies
	\begin{equation}
		\begin{aligned}
			&{\supp}(\hat{\varphi}_j)\subset \left\{\xi:2^{j-1}\leq |\xi|\leq 2^j\right\},
			\\ 
			&|\partial^{\alpha}\hat{\varphi}_j(\xi)|\leq C_{\alpha}2^{-j|\alpha|} \qquad\textit{for every multi-index $\alpha$,}
			\\
			&|\hat{\varphi}_j|>C\qquad\qquad \qquad\quad\textit{for~ ${3\over 5}\leq 2^{-j}|x|\leq {5\over 3}$}.
		\end{aligned}
	\end{equation}
	
	Frazier and Jawerth \cite{FJ} obtained two types of decompositions of Besov space, and both decompositions utilize a discrete Calder\'{o}n reproducing formula. The first type is an atomic decomposition. %which is similar to the atomic decomposition of Hardy  spaces $H^p(\R^n),~0<p\leq1$.
	%Let $a_Q$ denote the $(s,p)$-atoms (defined in \cite{FJ}) supported on $Q$. Let $Q^v$ be the collections of dyadic cubes with side length $2^{-v}$. The following norm equivalence  was established:
	%\begin{equation}
		%	\|f\|_{\dot{B}_{p,q}^s}\sim \inf \left\{\left(\sum_{v\in\Z}\left(\sum_{Q\in Q^v}|s_Q|^p\right)^{q\over p}\right)^{1\over q}:f=\sum_{v\in\Z}\sum_{Q\in Q^v}s_Qa_Q~ \textit{in $\mathcal{S}'/{\mathcal P}$}\right\}.
		%\end{equation}
	%And explicit formulas for $s_Q$ and $a_Q$ were founded for this decomposition, which leads to a discrete reproducing formula.
	The other decomposition  uses following building blocks $\psi_{\nu{\bf k}}$. Let $\psi $ be a Schwartz function satisfying that ${\rm\bf(i).}$ ${\rm supp}~\hat{\psi}\subset\{\xi\in \R^n:{1\over 2}\leq|\xi|\leq 2\}$, ${\rm\bf(ii).}$ $|\hat{\psi}(\xi)|\geq c>0$ for ${3\over 5}\leq|\xi|\leq {5\over 3}$, ${\rm\bf(iii).}$ $\sum\limits_{\nu\in\Z}\hat{\psi}(2^{\nu}\xi)\hat{\varphi}(2^{\nu}\xi)=1$ for all $\xi\neq0$. The building blocks are
	$$
	\psi_{Q_{\nu{\bf k}}}(x)=|Q_{\nu{\bf k}}|^{{\alpha\over n}-{1\over p}}\psi(2^{-\nu}x-\bf k), \qquad \nu\in\Z, ~{\bf k}\in\Z^n,
	$$
	where $Q_{\nu{\bf k}}=\prod_{i=1}^n{\bf \big[}2^{\nu}k_i,2^{\nu}(k_i+1){\bf\big]}$. The second decomposition utilizes the following discrete reproducing formula:
	\begin{equation}\label{FJ reproducing}
		f(x)=\sum_{\nu\in\Z}\sum_{{\bf k}\in\Z^n}2^{\nu n({\alpha\over n}-{1\over p})}\psi_{Q_{\nu{\bf k}}}(x)~\varphi_{\nu}*f(2^{\nu}{\bf k}),
	\end{equation}
	and the following norm equivalence was proved:
	\begin{equation}\label{FJ}
		\|f\|_{\dot{B}_{p,q}^s}\sim\left(\sum_{\nu\in\Z}2^{{\nu}s}\left(\sum_{{\bf k}\in\Z^n}|\varphi_{\nu}*f(2^{\nu}{\bf k})|^p|Q_{\nu\bf k}|\right)^{q\over p}\right)^{1\over q}.
	\end{equation}
	
	In  \cite{CW}, Coifman and Weiss introduced spaces of homogeneous type. The  Besov spaces on  spaces of homogeneous type have also been studied. The homogeneous Besov spaces were considered in \cite{H-M-Y} and \cite{6-12}, and the inhomogeneous  Besov spaces were studied in \cite{6-9}. In \cite{6}, the authors introduced a new norm and proved its equivalence to the norm in \cite{6-9}. Using this new norm, they successfully established the Besov space for $p,q\leq 1$ on spaces of homogeneous type.

	In this paper, we consider the Besov spaces in the Dunkl setting, where the Fourier transform no longer exists. The fundamental consideration in the Dunkl setting  is to establish a discrete Calder\'{o}n reproducing formula, see \cite{Hp}. It is noteworthy that the discrete Calder\'{o}n reproducing formula in \cite{Hp} derives from the Dunkl Poisson kernel, which involves  both the  Dunkl metric and the Euclidean metric.
	% in the Dunkl setting is  the foundation of our consideration.
	%The Dunkl theory  is a far reaching generalization  of Euclidean Fourier  analysis.  This  foundation can be tracked back by the influntial paper of Dunkl \cite{5-14}.
	
	Now we would like to introduce the framework of Dunkl setting in \cite{4}  (see also \cite{2}, \cite{7} and \cite{dunkl measure}). Consider the Euclidean space $\R^n$ with the scalar product $\langle x,y\rangle=\sum_{j=1}^{N}x_j y_j $ and  the corresponding norm $\|x\|^2=\langle x,x\rangle$. We set $B(x,r):= \{y\in\R^n: \| x-y\|<r \}$ and denote the Euclidean ball centered at $x$ with radius $r>0$.
	Let  $R$ be  a  root system in $\R^n$ normalized so that $\langle\alpha, \alpha\rangle=2$ for $\alpha \in R$ with $R_+$ a fixed positive subsystem, and  $G$ be the finite reflection group generated by the reflections $\sigma_{\alpha} (\alpha\in R)$, where $\sigma_{\alpha}(x)=x-\langle\alpha,x\rangle \alpha$  for $x\in \R^n.$
	Define the Dunkl metric 
	$$
	d(x,y):=\min_{\sigma\in G}\left\|\sigma(x)-y\right\|
	$$
	as  the distance between  two $G-$orbits  $\mathcal{O}(x)$ and $\mathcal{O}(y)$. Obviously, $d(x,y)\leq\|x-y\|,~d(x,y)=d(y,x)$ and $ d(x,y)\leq d(x,z)+d(z,y)$ for all $x,y,z\in \R^n$. Moreover, $\omega(B(x,r))\sim\omega(B(y,r))$ as  $d(x,y)\sim r$,  and $\omega(B(x,r))\leq\omega(\mathcal{O}(B(x,r)))\leq|G|~\omega(B(x,r))$,  where $\mathcal{O}(B(x,r))=\{y\in \R^n,~d(y,x)<r\}$.
	Let  $\kappa(\alpha)~(\alpha\in R)$  be a non-negative multiplicity function on $R$. The Dunkl measure is defined by
	$$
	d\omega(x):=\prod_{\alpha\in R}\left|\langle\alpha,x\rangle\right|^{{ \kappa}(\alpha)}dx.
	$$
	The number ~$N=n+\sum_{\sigma\in R} \kappa(\alpha)$ ~is called the homogeneous dimension of the system. %For $x\in \R^n$ and $r>0$, let $B(x,r)=\left\{y\in \R^n~:~ d(x,y)<r\right\}$ stand for the ball with center $x\in \R^n$ and radius $r>0$.  We set
	Then
	$$
	\omega(B(tx,tr))= t^N \omega(B(x,r))  \qquad\textit {for $ x\in\R^n, ~t,r>0,$}
	$$
	and
	\begin{equation}\label{measure of ball}
		\omega(B(x,r))\sim r^n \prod_{\alpha\in R}\left(\left|\langle\alpha,x\rangle\right|+r\right)^{{ \kappa}(\alpha)}\qquad\textit{for $x\in\R^n, r>0.$}
	\end{equation}
	It is important to indicate that for $x\in\R^n, ~r>0, \omega(B(x,r))\ge Cr^N.$
	
	Moreover, for $r_ 1, r_2>0,$  we have
	\begin{equation}\label{doubel condition}
		C^{-1}\left(r_2\over r_1\right)^n\leq  {{\omega(B(x,r_2))}\over {\omega(B(x,r_1))}} \leq C\left(r_2\over r_1\right)^N    \qquad\textit{for $0<r_1<r_2$}.
	\end{equation}
	This implies that $d{\omega(x)}$ satisfies  the doubling and reverse doubling properties, that is, there is a constant $C>0$ such that for all $x\in R^n, ~r>0$ and $\lambda\geq 1,$
	$$
	C^{-1} \lambda^n {\omega(B(x,r))}\leq {\omega(B(x,\lambda r))}\leq C\lambda^N {\omega(B(x,r))}.
	$$
	Let $E(x,y)$ be the associated Dunkl kernel, which was introduced in \cite{5-15}. It is known that $E(x,y)$  has a unique extension to a holomorphic function in $\mathbb C^N  \times\mathbb C^N$. The Dunkl transform, which generalizes  the classical Fourier transform on $R^n$,  is defined on $L^1(d\omega)$  by 
	$$
	{\bf\F} f(\xi)= c_{\kappa}^{-1} \int_{\R^n }f(x)E(x,-i\xi) d\omega(x),
	$$
	where  $c_{\kappa}= \int_{\R^n}e^{{-\|x\|^2}/2} d\omega(x)>0$. 
	The Dunkl transform  was introduced in \cite{4-7} for $\kappa\geq0$ and further studied in \cite{4-5} in a more general context, see also \cite{2-17}.  Particularly,  it satisfies the Plancherel identity,
	%is a topological  automophisms of $L^2(d\omega)$,
	i.e.,
	$$\|f\|_2=\|\F f\|_2$$
	for all $f\in L^2(d\omega)\cap L^1(d\omega)$ and the inversion formula  \cite{4-5}. %for every  $f\in L^1(d\omega)$, we have $f(x)= (\F)^2f(-x)$  ~for  all $x\in \R^n$.  For $\lambda >0,$  we have $\F (f_{\lambda})(\xi)= \F f(\lambda \xi)$, where $f_{\lambda}(x)=\lambda^{-N}f (\lambda^{-1}x)$
	If the multiplicity function $\kappa\equiv 0$, then the  Dunkl transform  boils down to the classical  Fourier transform.
	%$$
	%\hat{f}(\xi)= (2\pi)^{-n\over 2}\int_{\R^n} f(x) e^{-i\langle\xi,x\rangle} dx.
	%$$
	
	The classical Fourier transform  behaves well with the translation operator. However, the measure  $d\omega(x)$ is no longer  invariant  under the usual  translation. In \cite {4-20}, R\"{o}sler introduced the  Dunkl  translation $\tau_x f$ of $f$, which is defined  by
	$$
	{\bf\tau}_x f(-y)= c_{\kappa}^{-1}\int_{\R^n} E(i\xi, x)E(-i\xi, y)~\F f(\xi) d\omega(\xi)= \F^{-1}(E(i\cdot, x)\F f)(-y)
	$$	
	for   $f\in\mathcal{S} (\R^n)$	 and $x\in \R^n$ ( see also \cite{4-26}). It is easy to see that $\tau_x$ is bounded on $L^2(d\omega)$. However, it is still an open problem that if it is bounded operator on $L^p$ for $p \not=2.$
	
	% As an operator on $L^2(\R^n,\omega)$ $\tau_x$ is bounded. However, it is not all clear whether the translation operator  can be defined for $L^p$ functions with $p\neq 2$. Even the $L^p$ boundedness of $\tau_x$ on the dense subspace of Schwartz class for $p\neq 2$ is still open.
	For $f,g \in L^2(d\omega)$, the Dunkl convolution $f*g$ is defined  by the formula
	%$$f*g= c_{\kappa}~\F^{-1}\left((\F f)(\F g)  \right)(x)$$
	$$
	f*g(x)= \int_{R^n}f(y)~\tau_x g(-y) d\omega(y)= \int_{R^n} g(y)~\tau_x f(-y)d\omega(y).
	$$
	
	See \cite{4-20},  \cite{4-26}  and \cite{4-8} for more details related to Dunkl convolution.
	%Generalized convolusion of $f,g\in \mathcal{S}(\R^n)$ was considered in \cite{4-20} and \cite{4-8}, and the definition was extended to $f,g\in L^2(d\omega)$ in \cite{4-26}.
	
	The Dunkl operators  ${\bf T}_j$ were introduced in \cite{Dunkl},  which are defined by
	$$
	{\bf T}_j f(x)=\partial_j f(x)+\sum_{\alpha\in {R^+}} {\kappa(\alpha)\over 2}\langle\alpha,e_j\rangle {{f(x)-f(\sigma_{\alpha}(x))}\over \langle\alpha,x\rangle},
	$$
	where $e_1, e_2, \cdots, e_n$ are the standard unit vector of $\R^n$.
	
	The Dunkl Laplacian associated with $R$ and $\kappa$ is the operator $\Delta=\sum_{j=1}^n {\bf T}_j^2$, which is equivalent to
	$$
	\Delta f(x)= \Delta_{\R^n}f(x)+\sum_{\alpha\in R} \kappa(\alpha)\delta_\alpha f(x),
	$$
	where $\delta_\alpha f(x)={\partial_\alpha f(x)\over \langle\alpha,x\rangle}-{\|\alpha\|^2\over 2}{{f(x)-f(\sigma_{\alpha}(x))}\over {\langle\alpha,x\rangle}^2}$.
	
	The operator  $\Delta$ is essentially self-adjoint on $L^2(d\omega)$ (see \cite{4-1}), and generates the semigroup
	$$
	{\bf H}_t f(x)=e^{t\Delta}f(x)=\int_{\R^n} h_t(x,y)f(y) d\omega(y),
	$$
	where the heat kernel  $h_t(x,y)$ is a $C^\infty$ function for all $t>0,~x,y\in\R^n$ and satisfies
	$h_t(x,y)=h_t(y,x)>0$ and $\int_{\R^n}h_t(x,y)d\omega(y)=1.$
	
	The Poisson semigroup  ${\bf P}_t=e^{-t\sqrt{-\Delta}}$ is given by
	$$
	{\bf P}_t f(x)=\pi^{-{1\over2}}\int_0^\infty e^{-u}exp({t^2\over 4u}\Delta) f(x) {du\over{\sqrt u}},
	$$
	and solves the boundary value problem
	\[ \left\{
	\begin{array}{lr}
		(\partial_t^2+\Delta_x)u(x,t)=0,\\
		u(x,0)=f(x),
	\end{array}\right. \]
	in the half-space $\R^{n+1}_+ $  (see \cite{dunkl measure} and \cite{3-31}).
	We note that $\{{\bf P}_t\}_{t\in \R}$ is an approximation  to the identity on $L^2(d\omega)$, i.e.
	$\lim\limits_{t\to 0} {\bf P}_t(f)(x) = f(x)$  and~$\lim\limits_{t\to \infty} {\bf P}_t(f)(x) =0 ~  {\rm for} ~f\in L^2(d\omega).$

	Throughout this paper, a kernel of an operator  $\bf{G}$ is a distribution $G(x,y)$ for $x,y\in \R^n$ such that 
	$$
	{\bf G}f(x)=\int_{\R^n}G(x,y)f(y)d\omega(y).
	$$
	If an  operator is  represented by a letter in bold type, then its kernel will be represented by the same letter in regular type.

	Now we tend to discuss homogeneous Besov spaces in Dunkl setting. By the doubling and reverse doubling properties, $(\R^n, ||\cdot||, d\omega)$ can be viewed as a space   of   homogeneous type  in the sense of Coifman and Weiss.  Then   by \cite{H-M-Y}, the homogeneous Besov spaces ${\dot{B}_{p,cw}^{\alpha,q}} (d\omega)$ are established automatically.  In \cite{Bui}, the author established  Besov spaces ${\dot{B}_{p,q}^{\alpha,\Delta}}(d\omega)$ by the  Dunkl Laplacian $\Delta$,  which corresponds to the Dunkl metric $d(x,y)$, and proved the equivalence of ${\dot{B}_{p,q}^{\alpha,\Delta}}(d\omega)$ and ${\dot{B}_{p,cw}^{\alpha,q}}(d\omega).$

	In this paper, we give a new  characterization of Dunkl-Besov spaces.  The new ideas of our method are (i) utilize a new discrete Calder\'{o}n reproducing formula on $L^2(d\omega)$ in \cite{Hp}; (ii) introduce new test functions and distributions; (iii) establish a weak-type discrete Calder\'{o}n reproducing formula on Dunkl-Besov spaces.  

	Our motivation is the recent development of the Dunkl-Hardy space in \cite{Hp}, where the authors established a new discrete Calder\'{o}n reproducing formula by the Dunkl-Poisson kernel ${P}_t(x,y)$, which involves both the Euclidean metric and the Dunkl metric.  
More specifically, for  $k\in \Z$, let  $Q_d^k$ be the collections of dyadic cubes with side length $2^{-k-{M_0}}$ with some positive integer $M_0$.  Let 
	\begin{equation}
		{\bf D}_k:={\bf P}_{2^{-k}}-{\bf P}_{2^{-k-1}}
	\end{equation} 
	with the kernel $D_k(x,y)=P_{2^{-k}}(x,y)-P_{2^{-k-1}}(x,y)$,  and let 
	\begin{equation}
		{\bf D}_k^{M_0}:=\sum_{l=k-{M_0}}^{k+M_0} {\bf D}_l={\bf P}_{2^{-k+M_0}}-{\bf P}_{2^{-k-M_0-1}}
	\end{equation}
	with the kernel $D_k^{M_0}(x,y)=P_{2^{-k+M_0}}(x,y)-P_{2^{-k-M_0-1}}(x,y).$ In \cite{Hp}, the authors proved the following formula in $L^2(d\omega)$:
	\begin{equation}\label{fdk}
	f(x)=\sum_{k\in \Z}\sum_{Q\in Q_d^k} \omega(Q)D^{M_0}_k(x,x_Q){\bf D}_k(h)(x_Q),
	\end{equation}
	where the series converges in $L^2(d\omega), x_Q$  can be any fixed point in $Q\in Q_d^k$, $\|h\|_2\sim \|f\|_2,$ and $h$ is the preimage  of $f$ under the action of following invertible bounded  operator on $L^2(d\omega)$:
	\begin{equation}\label{T_{M_0}}
	{\bf T}_{M_0}(f)=\sum_{k\in \Z}\sum_{Q\in Q_d^k} \omega(Q)D^{M_0}_k(x,x_Q){\bf D}_k(f)(x_Q).
	\end{equation}
	The reproducing formula \eqref{fdk} leads to define the Dunkl-Besov norm for $f\in L^2(d\omega)$ as follows:
	\begin{equation}\label{Besov norm}
	\left\|f\right\|_{\dot{B}_{p,d}^{\alpha,q}}=\left\{\sum_{k\in\Z}2^{k\alpha q}\left\{\sum_{Q\in Q_d^k}\omega(Q)\left|{\bf D}_k(f)(x_Q)\right|^p\right\}^{q\over p}\right\}^{1\over q}.
	\end{equation}
	The first main result in this paper is the following
	
	\begin{thm}\label{Dunkl reproducing}
		Suppose that  $|\alpha|<1$,  $\max\left\{\frac{N}{N+1},\frac{N}{N+1+\alpha}\right\}< p<\infty$ and $0<q<\infty$. If $f\in L^2(d\omega)$ with ${\|f\|_{\dot{B}_{p,d}^{\alpha,q}}~<\infty}$, then there exists a function $h\in L^2(d\omega)$, such that $\|f\|_{2}\sim\|h\|_{2}$,  $\|f\|_{\dot{B}_{p,d}^{\alpha,q}}\sim\|h\|_{\dot{B}_{p,d}^{\alpha,q}}$ and
		\begin{equation}\label{f}
	f(x)=\sum_{k\in\Z}\sum_{Q\in Q_d^k} \omega(Q)D^{M_0}_k(x,x_Q){\bf D}_k(h)(x_Q),
	\end{equation}
		where the series converges in the $L^2(d\omega) $ norm and the Dunkl-Besov  space norm.
	\end{thm}
	
	Based on the above theorem, we prove the following duality estimates which are crucial for establishing the Besov spaces in the Dunkl setting. For this purpose, we first define the $dual\;indexes$ $\alpha',p',q'$ of $\alpha,p,q$ as follows:
	\begin{defn}\label{index}
		We define
		\[ \alpha'=
	\begin{cases}
			-\alpha,~~~~~~~~~~~~~~~~~~~~~ \textit{if  $1<p<\infty$},\\
			-\alpha+N({1\over p}-1),~~\textit{if $ 0<p\leq 1$},
		\end{cases}
		~~~~~~~p'=  \begin{cases}
			{p\over{p-1}},~~\textit{if  $1<p<\infty$},\\
			\infty,~~~~\textit{if  $0<p\leq 1$},
		\end{cases} \]
		and
		\[q'=  \begin{cases}
			{q\over{q-1}},~~\textit{if $1<q<\infty$},\\
			\infty,~~~~\textit{if $ 0<q\leq 1$}.
		\end{cases} \]
	\end{defn}

	Denote $L^2\cap\dot{B}_{p,d}^{\alpha,q}=\left\{ f\in L^2(d\omega):~\|f\|_{\dot{B}_{p,d}^{\alpha,q}}<\infty\right\}$. The duality estimates are given by the following:
	\begin{thm}\label{dual prop}
		Suppose  $|\alpha| <1$, $\max\left\{{N\over {N+1}},~{N\over{N+1+\alpha}}\right\}<p<\infty$ and $0<q<\infty$. If $f\in L^2\cap \dot{B}_{p,d}^{\alpha,q}$ and $g\in L^2\cap \dot{B}_{p',d}^{\alpha',q'}$,~then
		\begin{equation}\label{dual}
		\left|\langle f,g\rangle\right|~\leq~\left\|f\right\| _{\dot{B}_{p,d}^{\alpha,q}}    \left\|g\right\| _{\dot{B}_{p',d}^{\alpha',q'}}.
		\end{equation}
	\end{thm}

	By duality estimate, a function $f\in L^2\cap\dot{B}_{p,d}^{\alpha,q}$ can be viewed as a distribution on $L^2\cap \dot{B}_{p',d}^{\alpha',q'}$. Hence, let $L^2\cap \dot{B}_{p',d}^{\alpha',q'}$ be the space of new  test functions. We define the Besov space $\dot{\B}_{p,d}^{\alpha,q}(d\omega)$ to be the subspace of $(L^2\cap \dot{B}_{p',d}^{\alpha',q'})'$
	admitting a decomposition  with the building blocks $D_k^{M_0}(x,x_Q)$ in  \eqref{fdk}. More precisely, we will show the following:
	
	\begin{thm}\label{distribution converge}
		Suppose $|\alpha|<1$, $\max\left\{\frac{N}{N+1},\frac{N}{N+1+\alpha}\right\}<p<\infty$ and $0<q<\infty$. Suppose $\{f_n\}_{n=1}^{\infty}$ is a sequence in $L^2(d\omega)$ with
		$ \|f_n-f_m\|_{{\dot{B}_{p,d}^{\alpha,q}} }\to 0$ as $n,m\to \infty$. Then  there exist $f,h$ as distributions in $\left(L^2\cap\dot{B}^{\alpha',q'}_{p',d}\right)'$  with    $\|f\|_{\dot{B}^{\alpha,q}_{p,d}}\sim \|h\|_{\dot{B}^{\alpha,q}_{p,d}}$,  such that for each function $g\in L^2\cap\dot{B}^{\alpha',q'}_{p',d}$,
		\begin{enumerate}[(i)]
			\item
			$\lim\limits_{n\rightarrow\infty}\langle f_n, g\rangle=\langle f,g\rangle$;
			\item
			$\|f\|_{\dot{B}^{\alpha,q}_{p,d}}=\lim\limits_{n\rightarrow\infty}\|f_n\|_{\dot{B}^{\alpha,q}_{p,d}}$;
			\item
			$f$ has a following weak-type discrete Calder\'{o}n reproducing formula in the distribution sense:				
			\begin{equation}\label{weak type reproducing}
				\left\langle f,g\right\rangle=\left\langle\sum\limits_{k\in\Z}\sum\limits_{Q\in Q_d^k}\omega(Q)D_k^{M_0}(\cdot,x_Q){\bf D}_k(h)(x_Q), g\right\rangle
				=\sum_{k\in\Z}\sum_{Q\in Q_d^k} \omega(Q){\bf D}_k(h)(x_Q) {\bf D}_k^{M_0}(g)(x_Q),
			\end{equation}
	where the last series converges absolutely.
		\end{enumerate}
	\end{thm}
	
		It is ready to define the Dunkl-Besov space
	$\dot{\B}_{p,d}^{\alpha,q}(d\omega)$ as follows:
	\begin{defn}\label{inf}
		For $|\alpha|<1$, $\max\left\{\frac{N}{N+1},\frac{N}{N+1+\alpha}\right\}<p<\infty$ and $0<q<\infty$, the Dunkl-Besov space
		$\dot{\B}_{p,d}^{\alpha,q}(d\omega)$ is defined to be  the collection of all distributions $f\in(L^2\cap \dot{B}_{p',d}^{\alpha',q'})'$ such that
		$$
		f=\sum_{k\in\Z}\sum_{Q\in Q_d^k}\omega(Q)D_k^{M_0}(\cdot,x_Q)\lambda_Q
		$$
		with $\left\{\sum\limits_{k\in\Z} 2^{k\alpha q} \left\{\sum\limits_{Q\in Q_d^k}\omega (Q){\left|\lambda_{Q}\right| ^p}\right\}^{q\over p} \right\} ^{1\over q}<\infty$, where the series converges in the sense of distributions.
	If $f\in \dot{\B}_{p,d}^{\alpha,q}(d\omega)$, ~ the norm of $f$ is defined by
		\begin{equation}\label{finf}
			\left\|f\right\|_{\dot{\B}^{\alpha,q}_{p,d}}=\inf\left\{\left\{\sum_{k\in\Z} 2^{k\alpha q} \left\{\sum_{Q\in Q_d^k}\omega (Q){\left|\lambda_{Q}\right| ^p}\right\}^{q\over p} \right\} ^{1\over q}\right\},
		\end{equation}		
		where the infimum is taken over all $\{\lambda_Q\}$ such that $f(x)=\sum\limits_{k\in\Z}\sum\limits_{Q\in Q_d^k}\omega(Q)D_k^{M_0}(x,x_Q)\lambda_Q.$
	    \end{defn}
	
	Finally, we will prove	
	\begin{thm}\label{completness}
		For $\left|\alpha\right|<1$, $~max\left\{{N\over {N+1}},~{N\over{N+1+\alpha}}\right\}<p<\infty$ and $0<q<\infty$, we have
		\begin{equation}\label{thm4.1}
		\dot{\B}^{\alpha,q}_{p,d}=\bar{ L^2\cap\dot{B}^{\alpha,q}_{p,d}},
		\end{equation}
	where ~$\bar{ L^2\cap\dot{B}^{\alpha,q}_{p,d}}$  ~  is the  collection of all distributions  ~$f\in\left(L^2\cap\dot{B}^{\alpha',q'}_{p',d}\right)'$  ~such that there exists a sequence ~$\{f_n\}_{n=1}^{\infty}\in L^2({d\omega})$     ~with~ $\|f_n-f_m\|_{{\dot{B}_{p,d}^{\alpha,q}} }\to 0 $~ as $n, m\to \infty$,   and $f_n\to f$ as $n\to\infty$ in $\left(L^2\cap\dot{B}^{\alpha',q'}_{p',d}\right)'$.
	\end{thm}		
	As a  consequence, the Dunkl-Besov space $\dot{\B}^{\alpha,q}_{p,d}(d\omega)$ is a complete space, and we get the following equivalence of Besov space:
	\begin{equation}
		\dot{\B}_{p,d}^{\alpha,q}(d\omega) \approx  \dot{B}_{p,cw}^{\alpha,q}(d\omega) \approx \dot{B}_{p,q}^{\alpha,\Delta}(d\omega).
	\end{equation}

 As an application, we consider  the boundedness of the Dunkl-Calder\'{o}n-Zygmund operators on the Dunkl-Besov spaces.  The Dunkl-Calder\'{o}n-Zygmund singular integral operators was introduced in \cite{Hp}: let $C_0^{\eta},~\eta>0,$ denote the set of all  continuous functions $f$ with compact support and $\left\|f\right\|_{\eta}=\sup\limits_{x\neq y}\frac{\left|f(x)-f(y)\right|}{\left|x-y\right|^{\eta}}< \infty$;
 \begin{defn}
  	An operator ${\bf T}:C_0^\eta(\R^n)\rightarrow (C_0^\eta(\R^n))'$ with $\eta>0,$ is said to be a  Dunkl-Calder\'{o}n-Zygmund  singular integral operator  if $K(x,y)$, the kernel of ${\bf T}$, satisfies the following estimates: for some  $0<\varepsilon\leq1$
  	\begin{equation}\label{k1}
  		\left|K(x,y) \right|\leq C {1\over\omega(B(x,d(x,y))) } \left({d(x,y)}\over {||x-y||} \right)^\varepsilon
  	\end{equation}
  	for ~all ~$x\neq y;$
  	\begin{equation}\label{k2}
  		\left| K(x,y)-K(x',y) \right|\leq C  {1\over\omega(B(x,d(x,y))) } \left(||x-x'||\over {||x-y||} \right)^\varepsilon
  	\end{equation}
  	for~$||x-x'||\leq {d(x,y)/2};$
  	\begin{equation}\label{k3}
  		\left| K(x,y)-K(x,y') \right|\leq C {1\over\omega(B(x,d(x,y))) } \left(||y-y'||\over {||x-y||} \right)^\varepsilon
  	\end{equation}
  	for~$||y-y'||\leq {d(x,y)/2}.$
  	Moreover,
  	$$
  	\langle {\bf T}(f), g\rangle = \int_{\R^n}\int_{\R^n} K(x,y) f(x)g(y)d\omega(x)d\omega(y),
  	$$
  	for $\supp (f)\cap \supp (g)=\emptyset$. $\var$ is called the regularity exponent of ${\bf T}$. Define 
  	$$\|K\|_{dcz} :=\inf \left\{  C:\textit{\eqref{k1}- \eqref{k3} hold} \right\}.$$ The Dunkl-Calder\'{o}n-Zygmund operator norm is defined by
  	\begin{equation}
  	\|{\bf T}\|_{dcz}=\|{\bf T}\|_{2,2}+\|K\|_{dcz}.
  	\end{equation}
   \end{defn}
   A  Dunkl-Calder\'{o}n-Zygmund  singular integral operator  is said to be the  Dunkl-Calder\'{o}n-Zygmund  operator  if it extends  a bounded operator  on $L^2(d\omega)$. In \cite{Hp}, for this operator, the authors established the ${\bf T}1$ theorem and the boundedness on Dunkl-Hardy spaces.  Motivated by these results,  We obtain  the boundedness of the  Dunkl-Calder\'{o}n-Zygmund operators on the Dunkl-Besov spaces as follows: 
 
 \begin{thm}\label{T bound}
 	Let ${\bf T }f(x)=\int_{\R^n} K(x,y) f(y) d\omega(y)$ be a Dunkl-Calder\'{o}n-Zygmund operator. Suppose $\varepsilon_0$ is the regularity exponent of the kernel $K(x,y)$.  Suppose $\left|\alpha\right|<1$, $\max\left\{{N\over {N+1}},~{N\over{N+1+\alpha}}\right\}<p<\infty$ and $0<q<\infty$,
 	\begin{enumerate}[(i)]
 		\item
 		If ${\bf T}1=0$, then for $0<\alpha<\var_0$ and $ p>{N\over N+\alpha}$, ${\bf T}$ is bounded on $\dot{\B}_{p,d}^{\alpha,q}(d\omega)$;
 		\item
 		If ${\bf T^*}1=0$, then for $-\var_0<\alpha<0$ and $ p>\max\left\{{N\over N-\alpha},{N\over N+\var_0+\alpha}\right\}$, ${\bf T}$ is bounded on $\dot{\B}_{p,d}^{\alpha,q}(d\omega)$;
 		\item
 		If  ${\bf T}1={\bf T^*}1=0$, then for $|\alpha|<\var_0$ and $p>\max\left\{{N\over N+\var_0},{N\over N+\var_0+\alpha}\right\}$, ${\bf T}$  is bounded on $\dot{\B}_{p,d}^{\alpha,q}(d\omega)$.
 	\end{enumerate}
 \end{thm}

	The organization  of the paper is as follows.  In the next section, we will introduce the preliminary work required for this article. Section $3$ will provide the proof of  {\bf Theorem \ref{Dunkl reproducing}}. In Section $4$, we will prove the duality estimate {\bf Theorem \ref{dual prop}}. The proof of {\bf Theorem \ref{distribution converge}} and {\bf Theorem \ref{completness}} will be given in Section $5$. In Section 6, we will prove  {\bf Theorem \ref{T bound}}.

	\section{Preliminaries}
	\setcounter{equation}{0}
	
 The kernel $D_k(x,y)$ (and $D_k^{M_0}(x,y)$) in \eqref{Dunkl reproducing} satisfy the following estimates for each $0<\varepsilon\leq 1$ (see \cite{Hp}):
	\begin{equation}\label{size}
		\left|D_k(x,y)\right|\leq C\frac{1}{V(x,y,2^{-k}+d(x,y))}\left(\frac{2^{-k}}{2^{-k}+\|x-y\|}\right)^{\varepsilon};
	\end{equation} 
	\begin{equation} \label{regularity1}
		\begin{aligned}
			\left|D_k(x,y)-D_k(x',y)\right|\leq C\left(\frac{\left\|x-x'\right\|}{2^{-k}}\right)^{\varepsilon}&\left\{\frac{1}{V(x,y,2^{-k}+d(x,y))}\left(\frac{2^{-k}}{2^{-k}+\|x-y\|}\right)^{\varepsilon}\right.
			\\ 
			&\left.+\frac{1}{V(x',y,2^{-k}+d(x',y))}\left(\frac{2^{-k}}{2^{-k}+\|x'-y\|}\right)^{\varepsilon}\right\};
		\end{aligned}
	\end{equation}
	\begin{equation}\label{regularity2}
		\begin{aligned}
			\left|D_k(x,y)-D_k(x,y')\right|\leq C\left(\frac{\left\|y-y'\right\|}{2^{-k}}\right)^{\varepsilon}&\left\{\frac{1}{V(x,y,2^{-k}+d(x,y))}\left(\frac{2^{-k}}{2^{-k}+\|x-y\|}\right)^{\varepsilon}\right.
			\\ &\left.+\frac{1}{V(x,y',2^{-k}+d(x,y'))}\left(\frac{2^{-k}}{2^{-k}+\|x-y'\|}\right)^{\varepsilon}\right\};
		\end{aligned}
	\end{equation}
	\begin{equation}\label{cancelation} \int_{\R^n}D_k(x,y)d\omega(x)=\int_{\R^n}D_k(x,y)d\omega(y)=0.
	\end{equation}	
	
	If we fix the variable $x$ (or $y$), then $D_k(x,\cdot)$ (or $D_k(\cdot,y)$) is a smooth molecule (see \cite [{\bf Definition 1.6}]{Hp}). In this paper, a function $f(x)$ is called a smooth molecule with regularity $\eta$ if there exists $x_0\in \R^n$, such that the function $f(x)$ satisfies \eqref{size}, \eqref{regularity1} and \eqref{cancelation} only for the variable $x$, with $y=x_0$, $\var=\eta$ and some integer $k$.
	The following result implies that $D_k(\cdot,y),D_k(x,\cdot),D_k^{M_0}(\cdot,y),D_k^{M_0}(x,\cdot)$ all belong to $L^2\cap\dot{B}_{p,d}^{\alpha,q}$ for $|\alpha|<1$, $\max\left\{{N\over N+1},{N\over N+1+\alpha}\right\}<p<\infty$ and $0<q<\infty$.
	
	\begin{lemma}\label{basis}
		Suppose $f(x)$ is a smooth molecule with regularity $0<\eta\leq 1$. Then $f\in L^2\cap\dot{B}_{p,d}^{\alpha,q}$ for $|\alpha|<\eta$, $\max\left\{{N\over N+\eta},{N\over N+\eta+\alpha}\right\}<p<\infty$ and $0<q<\infty$.
		\begin{proof}
			Suppose $f(x)$ satisfies \eqref{size}, \eqref{regularity1} and \eqref{cancelation} only for the variable $x$, with $y=x_0$, $\var=\eta$ and some integer $k$. Then by   \cite[{\bf Lemma 2.11}]{Hp}, for any $\var<\eta$ we have
			\begin{equation}
				\begin{aligned}
					\left|D_{k'}(f)(x_{Q'})\right|&=\left|\int_{\R^n}D_{k'}(x,z)f(z)d\omega(z)\right|\\
				    &\leq C_{\var}2^{-|k-k'|\varepsilon}\frac{1}{V(x_{Q'},x_0,2^{-k\lor{-k'}}+d(x_{Q'},x_0))}\left(\frac{2^{{-k\lor{-k'}}}}{2^{-k\lor{-k'}}+d(x_{Q'},x_0)}\right)^{\varepsilon}\\
				    &\leq C_{\var}2^{-|k-k'|\varepsilon}\sum_{\sigma\in G}\omega(Q') \frac{1}{V(x_{Q'},\sigma(x_0),2^{-k\lor{-k'}}+\left\|x_{Q'}-x_0\right\|)}\left(\frac{2^{{-k\lor{-k'}}}}{2^{-k\lor{-k'}}+\left\|x_{Q'}-x_0\right\|}\right)^{\var},
			    \end{aligned}
			\end{equation}
		    where the notation $V(x,y,r)$ to denote the maximal value of $\omega(B(x,r))$ and $\omega(B(y,r))$, and the notation $a\lor b=\max\{a,b\}$. Choose $\var<\eta$ such that $|\alpha|<\var$ and  $p>\max\left\{{N\over N+\var},{N\over N+\var+\alpha}\right\}$. We have
		    \begin{equation}
		    	\begin{aligned}
		    	\|f\|_{\dot{B}_{p,d}^{\alpha,q}}^q=\sum_{k'\in\Z}2^{k'\alpha q}\left\{\sum_{Q'\in Q^{k'}_d}\omega(Q')\left|D_{k'}(f)(x_{Q'})\right|^p\right\}^{q\over p}.
		    	\end{aligned}
		    \end{equation}
		    We will use the fact that for each $\var'>0$, (see \cite[{\bf Lemma 2.2}]{Hp}) 
		    \begin{equation}\label{hehe}
		    	\sum_{Q'\in Q^{k'}_d}\omega(Q')\frac{1}{V(x_{Q'},x_0,2^{-k\lor{-k'}}+d(x_{Q'},x_0))}\left(\frac{2^{{-k\lor{-k'}}}}{2^{-k\lor{-k'}}+d(x_{Q'},x_0)}\right)^{\var'}<C_{\var'}.
		    \end{equation}
            Let $A_0=\left\{Q':\textit{$Q'\in Q^{k'}_d,\left\|x_{Q'}-\sigma(x_0)\right\|\leq 2^{-k\lor{-k'}}$}\right\}$, and for $l\geq 1$, 
            $$
            A_l= \left\{Q':\textit{$Q'\in Q_d^{k'}, 2^{-k\lor{-k'}+l-1}\leq \left\|x_{Q'}-\sigma(x_0)\right\|\leq 2^{-k\lor{-k'}+l}$}\right\}.
            $$ 
            Then for any $0<\var'<\var p$,
           \begin{equation}\label{haha}
           	\begin{aligned}
           	&\sum_{Q'\in Q^{k'}_d}\omega(Q')\left|D_{k'}(f)(x_{Q'})\right|^p
           	\\
           	&\leq C \sum_{\sigma\in G}\sum_{l=0}^{\infty}\sum_{Q'\in A_l}2^{-|k-k'|\var p}\left\{\omega(Q')\frac{1}{V(x_{Q'},\sigma(x_0),2^{-k\lor{-k'}}+\left\|x_{Q'}-\sigma(x_0)\right\|)}\left(\frac{2^{{-k\lor{-k'}}}}{2^{-k\lor{-k'}}+\left\|x_{Q'}-\sigma(x_0)\right\|}\right)^{\var'}\right\}\\
           	&~~~~~~~~~\qquad\qquad\times\left(\frac{2^{{-k\lor{-k'}}}}{2^{-k\lor{-k'}+l}}\right)^{\var p-\var'}V(x_{Q'},\sigma(x_0),2^{-k\lor{-k'}+l})^{1-p}.
           	\end{aligned}
           \end{equation}
           Let $Q_{k',\sigma}$ be a cube centered at $\sigma(x_0)$ with side length $2^{-k\lor{-k'}}$. Then  
           $$
           C^{-1}< V(x_{Q'},\sigma(x_0),2^{-k\lor{-k'}+l})\leq C2^{lN}\omega(Q_{k',\sigma}).
           $$
           By \eqref{hehe} and $\eqref{haha}$, if $p\geq 1$, then
           $$
           \sum_{Q'\in Q^{k'}_d}\omega(Q')\left|D_{k'}(f)(x_{Q'})\right|^p\leq C2^{-|k-k'|\var}.
           $$
           Because $|\alpha|<\var$, we have
           $$
           \|f\|^q_{\dot{B}_{p,d}^{\alpha,q}}\leq C \sum_{k'\in\Z}2^{(k'\alpha -|k-k'|\var)q}<\infty.
           $$
           If $p<1$, then
           $$
          \left(\frac{2^{{-k\lor{-k'}}}}{2^{-k\lor{-k'}+l}}\right)^{\var p-\var'}V(x_{Q'},\sigma(x_0),2^{-k\lor{-k'}+l})^{1-p}\leq 2^{(N(1-p)-\var p-\var')l}\omega(Q_{k',\sigma})^{1-p}.
           $$
           Because $p>{N\over N+\var}$, we can choose $\var'$ such that $N(1-p)-\var p-\var'<0$. By \eqref{hehe} and \eqref{haha},
           \begin{equation}
           	\sum_{Q'\in Q^{k'}_d}\omega(Q')\left|D_{k'}(f)(x_{Q'})\right|^p\leq C\sum_{\sigma\in G}\omega(Q_{k',\sigma})^{1-p}2^{-|k-k'|\var p}.
           \end{equation}
           Hence
           \begin{equation}
           	\|f\|^q_{\dot{B}_{p,d}^{\alpha,q}}\leq \sum_{k'\in\Z}\left(\sum_{\sigma\in G}\omega(Q_{k',\sigma})^{{1\over p}-1}2^{(k'\alpha -|k-k'|\var)p}\right)^{q\over p}.
           \end{equation}
           Because $\omega(Q_{k',\sigma})\leq C2^{-k'N}$ for $k'<k$ and $\omega(Q_{k',\sigma})\leq C$ for $k'\geq k$, we have
           \begin{equation}
           	\begin{aligned}
           		\|f\|^q_{\dot{B}_{p,d}^{\alpha,q}}&\leq C \sum_{k'\geq 0}\omega(Q_{k',\sigma})^{{q\over p}-q}2^{(k'\alpha -|k-k'|\var)q}+C \sum_{k'<0}\omega(Q_{k',\sigma})^{{q\over p}-q}2^{(k'\alpha -|k-k'|\var)q}\\
           		&\leq C\sum_{k'\geq 0}2^{(k'\alpha -|k-k'|\var)q}+C\sum_{k'<0}2^{(k'\alpha -|k-k'|\var-k'N({1\over p}-1))q}.
           	\end{aligned}
           \end{equation}
           Because $|\alpha|<\var$ and $p>{N\over N+\alpha+\var}$, both above series converge. So $\|f\|^q_{\dot{B}_{p,d}^{\alpha,q}}<\infty$.
		\end{proof}
	\end{lemma}
	
	As in \cite{Hp}, we need introduce the discrete  Calder\'{o}n reproducing formula in spaces of homogeneous type (see \cite[{\bf Theorem 2.14}]{Hp}):
	\begin{theorem*}
		Let $\{{\bf S}_k\}_{k\in\Z}$ be  a Coifman's approximations to the identity. Set  ${\bf E}_k:={\bf S}_k-{\bf S}_{k-1}$. Then there~exists~a~family~of~operators $\{\widetilde{\bf E}_k\}_{k\in\Z}$ such~that~for~any fixed $x_Q\in Q$ with $k \in\Z$ and $Q\in Q_{cw}^k$,
		\begin{equation}\label{CW reproducing}
			f(x)= \sum_{k\in\Z}\sum_{Q\in Q_{cw}^k}\omega(Q){\widetilde E}_k(x,x_Q){\bf E}_k(f)(x_Q),
		\end{equation}
		where $Q_{cw}^k$ are collections of dyadic~cubes~with~the~side~length $2^{-{M_1}-k}$ and ~the~series~converge~in $L^p(d\omega)$, $1<p<\infty$, and moreover, for $\var=1$, the  kernel $E_k(x,y)$ satisfies \eqref{size}-\eqref{cancelation}, and the kernel $\widetilde{E}_k(x,y)$ satisfy \eqref{size}, \eqref{regularity1} and $\eqref{cancelation}$.
		\end{theorem*}

	%\begin{lemma}\label{almost orthogonal estimate}
		%	Let   $x,y \in \R^n$ and $\var,t,s>0$ with $t\geq s$. Suppose $f_t(x,\cdot)$ and $g_s(\cdot,y)$ both  are smooth molecule  functions in %${\M}_{\var_0,\var_0,t,x}$  and ${\M}_{\var_0,\var_0,s,y}$, respectively, then for any $0<\var_1,\var_2<\var_0$, there exists $C>0$ depending on %$\var_0,\var_1,\var_2$, such  that for all $t,s>0,$
		%	\begin{equation}
			%		\int_{\R^n} f_t(x,u) g_s(u,y)~ \omega(u)du~\leq~C~\left({s\over t }\wedge{t\over s}  \right)^{\var_1} {1\over {V(x,y,(t\lor s)+d(x,y))}}\left((t\lor %s)\over {(t\lor s)+d(x,y)}\right)^{\var_2}
			%	\end{equation}
		%	where  $a\wedge b= min\{a,b\} $  and  $a\vee b = max\{a,b\}$
		%\end{lemma}
	
	%In \cite{Hp}, the authors also establish the following ${\bf T}1$ theorem:
%\begin{theorem*}
%	Suppose that  ${\bf T}$  is a  Dunkl-Calder${\rm\acute{o}}$n Zygmund singular integral   operator. Then ${\bf T}$ extends  to  a bounded operator on $L^2(d\omega) $ if  and  only if \ (a) ${\bf T}(1)(x)\in BMO_{d}(d\omega)$;   (b) ${\bf T}^*(1)(x)\in BMO_{d}(d\omega);$   (c) ${\bf T} \in WBP$.
%\end{theorem*}
	
	The following almost  orthogonal estimate  in \cite{Hp} is crucial.
	\begin{lemma}\label{almost orthogonal estimate with dcz}
		Let ${\bf T}$ be  a Dunkl-Calder\'{o}n-Zygmund operator with ${\bf T}1={\bf T}^*1=0$.  Then
		\begin{equation}
		\begin{aligned}
			&\left|\int_{\R^n} \int_{\R^n}  E_k(x,u)K(u,v)\widetilde{E}_j(v,y)  d\omega(u) d\omega(v)\right|
			\\ 
			&\leq~C~2^{-|k-j|\var'}~\|{\bf T}\|_{dcz}~ {1\over {V(x,y,{2^{-j\lor-k}}+d(x,y))}}\left(({2^{-j\lor-k}}\over {2^{-j\lor-k}+d(x,y)}\right)^{\gamma},
		\end{aligned}
		\end{equation}
		where $\gamma,\var' \in (0,\var)$ and $\var$ is the regularity exponent of the kernel of ${\bf T}$  given in \eqref{k2}  and \eqref{k3}.
	\end{lemma}

	\section{ Discrete reproducing formula in $ L^2\cap \dot{B}_{p,d}^{\alpha,q}$}
	\setcounter{equation}{0}
	In this section, we prove {\bf Theorem \ref{Dunkl reproducing}}. Firstly, we recall the sketch of the proof of formula \eqref{fdk} in $\cite{Hp}$. The key point is to prove the operator ${\bf T}_{M_0}$ (see \eqref{T_{M_0}}) is invertible on $L^2(d\omega)$ and ${({\bf T}_{M_0})}^{-1}$, the inverse of ${\bf T}_{M_0}$,  is bounded on $L^2(d\omega)$. For this, the author considered the operator
	$$
	{\bf R}_{M_0}(f):=f-{\bf T}_{M_0}(f),
	$$
	and proved that the $L^2$-norm of ${\bf R}_{M_0}$ is less than $1$. Then, the invertibility of ${\bf T}_{M_0}$ and the boundedness of $({\bf T}_{M_0})^{-1}$ follow. To show  that \eqref{fdk}) holds with respect to the Dunkl-Hardy space norm, the authors used the similar strategy.
	
	In this paper, the key to proving  {\bf Theorem \ref{Dunkl reproducing}} is to prove that ${({\bf T}_{M_0})}^{-1}$ is bounded in the Dunkl-Besov norm \eqref{Besov norm}. Hence, we need  estimate ${\bf R}_{M_0}$ on $L^2\cap \dot{B}_{p,d}^{\alpha,q}$ and show that the norm of ${\bf R}_{M_0}$ on $L^2\cap \dot{B}_{p,d}^{\alpha,q}$ is less than $1$. To this end, we introduce the following norm as an intermediary:
	\begin{equation}
		\|f\|_{\dot{B}_{p,cw}^{\alpha,q}}:=\left\{\sum_{k\in\Z}\left\|2^{k\alpha}\sum_{Q\in Q_{cw}^k}{\bf E}_k(f)(x_Q)\chi_Q(x) \right\|_p^q \right\}^{\frac{1}{q}}.
	\end{equation}

	Now we start our proof. The following estimate is crucial.
	\begin{lemma}\label{iterate}
		Let $k,k',M,M'$ belong to $\Z$, $\var>0,$ $|\alpha|<\varepsilon$,~ $\max\left\{\frac{N}{N+\varepsilon},\frac{N}{N+\varepsilon+\alpha}\right\}< p\leq\infty$ and $0<q\leq\infty$. Suppose that $S_{k,k'}(x,y)$ satisfies the following condition:
		$$
		\left|S_{k,k'}(x,y)\right|\leq C~2^{-|k-k'|\varepsilon}\frac{1}{V(x,y,2^{-k\lor{-k'}}+d(x,y))}\left(\frac{2^{{-k\lor{-k'}}}}{2^{-k\lor{-k'}}+d(x,y)}\right)^{\varepsilon}.
		$$
		Then for collections of dyadic  cubes $Q_1^{k'}$ with side length $2^{-k'-{M'}}$ and $Q_2^k$ with side length $2^{-k-M}$, a series of real numbers $\{\lambda_Q\}_{k\in \Z,Q\in Q_2^k}$, and a constant $\theta$ with $\max\left\{\frac{N}{N+\varepsilon},\frac{N}{N+\varepsilon+\alpha}\right\}<\theta\leq 1$ and $\theta<p$,  we have
		$$
		\left\{\sum_{k'\in\Z}\left\|2^{k'\alpha}\sum_{Q'\in Q_1^{k'}}\sum_{k\in\Z}\sum_{Q\in Q_2^k}\omega(Q)S_{k,k'}(x_{Q'},x_{Q})\lambda_{Q}\chi_{Q'}(x)\right\|_{p}^q \right\}^{\frac{1}{q}}\leq C \left\{\sum_{k\in\Z}\left\|2^{k\alpha}\sum_{Q\in Q_2^k}\lambda_Q\chi_Q(x) \right\|_p^q \right\}^{\frac{1}{q}},
		$$
		for any fixed points $x_{Q'}$ in $Q'\in Q_1^{k'}$ and $x_{Q}$ in $Q\in Q_{2}^{k}$, where the constant  $C$ only depends on $N,\alpha,\var,\theta,{p\over\theta},M$.
	\end{lemma}
 \begin{proof}
  By \eqref{doubel condition}, for $x,y\in\R^n$ and radius $r$, if $d(x,y)\lesssim r$, then we have
	$$
	V(x,y,r)\sim \omega(B(x,r))\sim \omega(B(y,r)).
	$$
	If $x$ belongs to $Q'\in Q_1^{k'}$, then $\left\|x-x_{Q'}\right\|\leq 2^{-k'-M'}$. So for $x_Q\in Q_2^k$ and $x,x_{Q'}\in Q_1^{k'}$, we have $\left\{2^{-k\lor -k'}+d(x_{Q'},x_Q)\right\}\sim \left\{2^{-k\lor -k'}+d(x,x_Q)\right\}$. By the definition of $d(x,y)$,
	\begin{equation}\label{AA}
		\begin{aligned}
			&|S_{k,k'}(x_{Q'},x_{Q})|\chi_{Q'}(x)
			\\
			&\leq~ C~2^{-|k-k'|\varepsilon}\frac{1}{V(x,x_Q,2^{{-k\lor{-k'}}}+d(x,x_Q))}\left(\frac{2^{{-k\lor{-k'}}}}{2^{{-k\lor{-k'}}}+d(x,x_Q)}\right)^{\varepsilon}\chi_{Q'}(x)
			\\
			&\leq~\sum_{\sigma\in G}C~2^{|k-k'|\varepsilon}\frac{1}{V(\sigma(x),x_Q,2^{{-k\lor{-k'}}}+\|\sigma(x)-x_Q\|)}\left(\frac{2^{{-k\lor{-k'}}}}{2^{{-k\lor{-k'}}}+\|\sigma(x)-x_Q\|}\right)^{\varepsilon}\chi_{Q'}(x).
		\end{aligned}
	\end{equation}
	
	Because $0<\theta\leq 1$, by the $\theta$-triangle inequality $|a+b|^{\theta}\leq |a|^{\theta}+|b|^{\theta}$,
	\begin{equation}\label{DDD}
		\begin{aligned}
			&\left\|2^{k'\alpha }\sum_{k\in\Z}\sum_{Q'\in Q_1^{k'}}\sum_{Q\in Q_2^k}\omega(Q)~S_{k,k'}(x_{Q'},x_{Q})\lambda_Q\chi_{Q'}(x)\right\|_p
			\\
			&\leq~\left\|2^{k'\alpha\theta}\sum_{k\in\Z}2^{-|k-k'|\var\theta} \sum_{\sigma\in G}\sum_{Q'\in Q_1^{k'}}  \sum_{Q\in Q_2^k}\frac{\omega(Q)^{\theta}}{V(\sigma(x),x_Q,2^{{-k\lor{-k'}}}+\|\sigma(x)-x_Q\|)^{\theta}}~~\right.
			\\
			&\left.~~~~~~~~~~~~~~~~~~~~~~~~~~~~~~~~~~~~~~\times~\left(\frac{2^{{-k\lor{-k'}}}}{2^{{-k\lor{-k'}}}+\|\sigma(x)-x_Q\|}\right)^{\varepsilon\theta}|\lambda_Q|^{\theta}\chi_{Q'}(x)\right\|_{p\over\theta}^{1\over \theta}.
		\end{aligned}
	\end{equation}	

	We now estimate
	$$
	I:=\sum_{Q\in Q_2^k} \frac{\omega(Q)^{\theta}}{V(\sigma(x),x_Q,2^{{-k\lor{-k'}}}+\|\sigma(x)-x_Q\|)^{\theta}}\left(\frac{2^{{-k\lor{-k'}}}}{2^{{-k\lor{-k'}}}+\|\sigma(x)-x_Q\|}\right)^{\varepsilon\theta} |\lambda_Q|^\theta.
	$$
	Let
	$$
	A_0:=\left\{Q:\;Q\in Q_2^k, \;\|\sigma(x)-x_Q\|\leq 2^{{-k\lor{-k'}}}\right\},
	$$
	and for $l\geq 1$,
	$$
	A_l:=\left\{Q:\;Q\in Q_2^k,\;2^{{l-1+(-k\lor{-k'})}}\leq\|\sigma(x)-x_Q\|\leq 2^{l+(-k\lor{-k'})}\right\}.
	$$
	Because $\omega(Q)\sim\omega(B(x_Q,2^{-k}))$, we have
	\begin{equation}
		\begin{aligned}
			I\leq C\sum_{l=0}^{\infty}\sum_{Q\in A_l}\left(\frac{\omega(B(x_Q,2^{-k}))}{\omega(B(x_Q,2^{l+{(-k\lor{-k'})}}))}\right)^{\theta-1}\frac{\omega(Q)}{\omega(B(\sigma(x),2^{l+(-k\lor{-k'})}))}\left(\frac{2^{-k\lor{-k'}}}{2^{l+(-k\lor{-k'})}}\right)^{\theta\varepsilon}|\lambda_{Q}|^{\theta}.
		\end{aligned}
	\end{equation}
	Using \eqref{doubel condition}, $\omega(Q)=\int_{\R^n}\chi_Q~{\rm d\omega(x)}$, and the Hardy-Littlewood maximal  operator ${\bf  M}$ on the  space $(\R^n, \|\cdot\|,d\omega)$,
	\begin{equation}\label{aa}
		\begin{aligned}
			I
			\leq &C~\sum_{l=0}^{\infty}
			\left(\frac{\omega(B(x_Q,2^{-k}))}{\omega(B(x_Q,2^{l+{(-k\lor{-k'})}}))}\right)^{\theta-1}\left(\frac{2^{-k\lor{-k'}}}{2^{l+(-k\lor{-k'})}}\right)^{\theta\varepsilon}
			\\ \ds
			&~~~~~~~~~~~~\times~~\frac{1}{\omega(B(\sigma(x),2^{l+(-k\lor{-k'})}))}\int_{\|y-\sigma(x)\|\leq 2^{l+(-k\lor{-k'})}}\sum_{Q\in Q_2^k}|\lambda_{Q}|^{\theta}\chi_Q(y)~d\omega(y)
			\\
			&	\leq C~2^{[-k-(-k\lor{-k'})]N(\theta-1)}\sum_{l=0}^{\infty} 2^{-lN(\theta-1)}2^{-l\theta\varepsilon}{\bf M}\left(\sum_{Q\in Q_2^k}|\lambda_{Q}|^{\theta}\chi_Q\right)(\sigma(x))
			\\
			&\leq C~2^{[-k-(-k\lor{-k'})]N(\theta-1)}{\bf M}\left(\sum_{Q\in Q_2^k}|\lambda_{Q}|^{\theta}\chi_Q\right)(\sigma(x)),
		\end{aligned}
	\end{equation}
	where the last inequality holds because $\theta>\frac{N}{N+\varepsilon}$.
	
	Note that ${p\over \theta}>1$. Then by \eqref{DDD}, \eqref{aa} and Minkowski inequality,
	\begin{equation}
		\begin{aligned}
			&\left\|2^{k'\alpha }\sum_{k\in\Z}\sum_{Q'\in Q_1^{k'}}\sum_{Q\in Q_2^k}\omega(Q)~S_{k,k'}(x_{Q'},x_{Q})\lambda_Q\chi_{Q'}(x)\right\|_p
			\\ 
			&\leq C \left\{\sum_{k\in\Z}2^{(k'-k)\alpha\theta-|k-k'|\varepsilon\theta}2^{[-k-(-k\lor{-k'})]N(\theta-1)}2^{k\alpha\theta} \sum_{\sigma\in G} \left\|{\bf M}\left(\sum_{Q\in Q_2^k}|\lambda_{Q}|^{\theta}\chi_Q\right)(\sigma(x))\right\|_{\frac{p}{\theta}}\right\}^{\frac{1}{\theta}}. 		
		\end{aligned}
	\end{equation}	
	Since $G$ is a finite group and $\int_{\R^n} f({\sigma(x)})~d\omega(x)=\int_{\R^n} f(x)~d\omega(x)$, we have
	\begin{equation}
		\begin{aligned}
			&\left\{\sum_{k'\in\Z}\left\|2^{k'\alpha }\sum_{Q'\in Q_1^{k'}}\sum_{k\in\Z}\sum_{Q\in Q_2^k}\omega(Q)S_{k,k'}(x_{Q'},x_{Q})\lambda_{Q}\chi_{Q'}(x)\right\|_{p}^q \right\}^{\frac{1}{q}}
			\\ 
			&\leq C\left\{\sum_{k'\in\Z}\left\{\sum_{k\in\Z}2^{{(k'-k)\alpha\theta-|k-k'|\varepsilon\theta}+{[-k-(-k\lor{-k'})]N(\theta-1)}}~2^{k\alpha\theta}\left\|{\bf M}\left(\sum_{Q\in Q_2^k}|\lambda_{Q}|^{\theta}\chi_Q\right)(x)\right\|_{\frac{p}{\theta}}\right\}^{\frac{q}{\theta}}
			\right\}^{\frac{1}{q}}.
		\end{aligned}
	\end{equation}
	Let $b_{k,k'}=2^{{(k'-k)\alpha\theta-|k-k'|\varepsilon\theta}+{[-k-(-k\lor{-k'})]N(\theta-1)}}$ and $g_k=2^{k\alpha\theta}\left\|{\bf M}\left(\sum\limits_{Q\in Q_2^k}|\lambda_{Q}|^{\theta}\chi_Q\right)(x)\right\|_{\frac{p}{\theta}}$.  Then we have
	\begin{equation}\label{bb}
		\left\{\sum_{k'\in\Z}\left\|2^{k'\alpha }\sum_{Q'\in Q_1^{k'}}\sum_{k\in\Z}\sum_{Q\in Q_2^k}\omega(Q)S_{k,k'}(x_{Q'},x_{Q})\lambda_{Q}\chi_{Q'}(x)\right\|_{p}^q \right\}^{\frac{1}{q}}\leq C\left\{\sum_{k'\in\Z}\left(\sum_{k\in\Z}b_{k,k'}g_k\right)^{q\over\theta}\right\}^{1\over q}.
	\end{equation}
	Because ${p\over \theta}>1$, by maximal function inequality,  we have
	\begin{equation}\label{bkgkM}
	\begin{aligned}
		\sum_{k\in\Z}\left(g_k\right)^{\frac{1}{\theta}}
		\leq C \sum_{k\in \Z}2^{k\alpha }\left\|\sum_{Q\in Q_2^k}|\lambda_Q|^{\theta}\chi_Q(x)\right\|_{\frac{p}{\theta}}^{\frac{1}{\theta}}
		=C\sum_{k\in\Z}2^{k\alpha}\left\|\sum_{Q\in Q_2^k}|\lambda_Q|\chi_Q(x)\right\|_{p}.
	\end{aligned}
	\end{equation}
	
	If $1<\frac{q}{\theta}<\infty$, then by \eqref{bb} and H\"{o}lder's inequality, we have
	\begin{equation}
		\begin{aligned}
			&\left\{\sum_{k'\in\Z}\left\|2^{k'\alpha }\sum_{Q'\in Q_1^{k'}}\sum_{k\in\Z}\sum_{Q\in Q_2^k}\omega(Q)S_{k,k'}(x_{Q'},x_{Q})\lambda_{Q}\chi_{Q'}(x)\right\|_{p}^q \right\}^{\frac{1}{q}}
			\\
			&\leq~C\left\{\sum_{k'\in\Z}\left(\sum_{k\in\Z}b_{k,k'}g_k\right)^{q\over\theta}\right\}^{1\over q}
			~\leq~C\left\{\sum_{k'\in\Z}~\left(\sum_{k\in\Z}b_{k,k'}\right)^{{q\over\theta} -1}\left(\sum_{k\in\Z} b_{k,k'}g_k^{\frac{q}{\theta}}\right)\right\}^{1\over q}.
		\end{aligned}
	\end{equation}
	Since  $\frac{N}{N+\varepsilon+\alpha}<\theta\leq 1$,   one can check that  $\sup\limits_{k'\in \Z}\left\{\sum\limits_{k\in\Z}b_{k,k'}\right\}<C$ and $\sup\limits_{k\in\Z}\left\{\sum\limits_{k'\in\Z}b_{k,k'}\right\} <C$.  Then we get that
	\begin{equation}\label{xjq1}
		\begin{aligned}		
			&\left\{\sum_{k'\in\Z}~\left(\sum_{k\in\Z}b_{k,k'}\right)^{{q\over\theta} -1}\left(\sum_{k\in\Z} b_{k,k'}g_k^{\frac{q}{\theta}}\right)\right\}^{1\over q}
			\\ 
			&\leq ~C~\left\{ \sum_{k'\in\Z}\sum_{k\in\Z} b_{k,k'}g_k^{q\over\theta} \right\}^{1\over q}
			= \left\{ \sum_{k\in\Z}\left(\sum_{k'\in\Z} b_{k,k'}\right)g_k^{q\over\theta} \right\}^{1\over q}
			\\
			&\leq~C\left\{\sum_{k\in\Z}\left\| 2^{k\alpha}\sum_{Q\in Q_2^k}|\lambda_Q|\chi_Q(x)\right\|_{p}^{q}\right\}^{\frac{1}{q}}.
		\end{aligned}
	\end{equation}
	If $\frac{q}{\theta}\leq 1$,  by \eqref{bb},
	\begin{equation}
		\begin{aligned}
			&\left\{\sum_{k'\in\Z}\left\|2^{k'\alpha }\sum_{Q'\in Q_1^{k'}}\sum_{k\in\Z}\sum_{Q\in Q_2^k}\omega(Q)S_{k,k'}(x_{Q'},x_{Q})\lambda_{Q}\chi_{Q'}(x)\right\|_{p}^q \right\}^{\frac{1}{q}}
			\\ 
			&\leq~C\left\{\sum_{k'\in\Z}\left(\sum_{k\in\Z}b_{k,k'}g_k\right)^{q\over\theta}\right\}^{1\over q}
			~\leq~ C\left\{\sum_{k'\in\Z}\sum_{k\in\Z}\left(b_{k,k'}\right)^{\frac{q}{\theta}}\left(g_k\right)^{\frac{q}{\theta}}\right\}^{1\over q}
			\\
			&= \left\{\sum_{k\in\Z}\left(\sum_{k'\in\Z}\left(b_{k,k'}\right)^{\frac{q}{\theta}}\right)\left(g_k\right)^{\frac{q}{\theta}}\right\}^{1\over q}
			~\leq~ C\left\{\sum_{k\in\Z}\left\| 2^{k\alpha}\sum_{Q\in Q_2^k}|\lambda_Q|\chi_Q(x)\right\|_{p}^{q}\right\}^{\frac{1}{q}}.
			\end{aligned}
	\end{equation}

	If $q=\infty$, then by \eqref{bb},
	\begin{equation}
		\begin{aligned}
			&\sup_{k'\in\Z}~\left\| 2^{k'\alpha }\sum_{Q'\in Q_1^{k'}}\sum_{k\in\Z}\sum_{Q\in Q_2^k}\omega(Q)S_{k,k'}(x_{Q'},x_{Q})\lambda_{Q}\chi_{Q'}(x)\right\|_{p}
			\\
			&\leq C~ \sup_{k'\in\Z}\left(\sum_{k\in\Z}b_{k,k'}g_k\right)^{1\over\theta}
			\leq  C~ \sup_{k'\in\Z}\left\{\sum_{k\in\Z}b_{k,k'}\right\}^{1\over\theta}\sup_{k\in\Z}\left(g_k\right)^{1\over\theta}
			\\ 
			&\leq C~ \sup_{k\in\Z}\left\| 2^{k\alpha }\sum_{Q\in Q_2^k}|\lambda_Q|\chi_Q(x)\right\|_{p}.
		\end{aligned}
	\end{equation}
\end{proof}

	Now we prove $({\bf T}_{M_0})^{-1}$ is bounded on $L^2\cap \dot{B}_{p}^{\alpha,q}$. 
	\begin{lemma}\label{Tbounded}~
		Suppose $|\alpha|<1$,~ $\max\left\{\frac{N}{N+1},\frac{N}{N+1+\alpha}\right\}< p\leq\infty$ and $0<q\leq\infty$.  Then ${\bf T}_{M_0}$ and  $({\bf T}_{M_0})^{-1}$ are bounded on   $L^2\cap \dot{B}_{p,d}^{\alpha,q}$.
	\end{lemma}
\begin{proof}
    Since  the $L^2$-boundedness of  ${\bf T}_{M_0}$ and $({\bf T}_{M_0})^{-1}$ has been proved by \cite{Hp}, here we only  show the boundedness in Dunkl-Besov norm.
	First, we will prove ${\bf T}_{M_0}$ and $({\bf T}_{M_0})^{-1}$ are bounded on the norm $\|\cdot\|_{\dot{B}_{p,cw}^{\alpha,q}}$ (see \eqref{CW norm}), then we prove that $\|\cdot\|_{\dot{B}_{p,d}^{\alpha,q}}\sim \|\cdot\|_{\dot{B}_{p,cw}^{\alpha,q}}$.  In this paper, for two operators $\bf{T_1}$ and $\bf{T_2}$ with kernels  $T_1(x,y)$ and $T_2(x,y)$, the kernel  of $\bf{T_1}\circ \bf{T_2}$ is denoted by $T_1T_2(x,y)$.

	{\bf {Step 1}:}
	${\bf T}_{M_0}$ and $({\bf T}_{M_0})^{-1}$ are bounded on the norm $\|\cdot\|_{\dot{B}_{p,cw}^{\alpha,q}}$.
	To show this estimate, the key point is to write  $f ={{\bf T}_{M_0}}(f)+{\bf R}_{M_0}(f)$. In \cite{Hp}, the authors proved that ${\bf R}_{M_0}$ is a Dunkl-Calder\'{o}n-Zygmund operator with $\left\|{\bf R}_{M_0}\right\|_{dcz}\leq C\ 2^{-M_0\delta}$ for some constant $\delta>0$. We claim that
	\begin{equation}\label{claim}
	\left\|{\bf R}_{M_0}(f)\right\|_{\dot{B}_{p,cw}^{\alpha,q}}\leq C\left\|{\bf R}_{M_0}\right\|_{dcz}\left\|f\right\|_{\dot{B}_{p,cw}^{\alpha,q}},
	\end{equation}
	where the constant $C$ is independent of $M_0$. After proving the claim \eqref{claim}, we will choose $M_0$ such that $C\left\|{\bf R}_{M_0}\right\|_{dcz}<{1\over 2}$, which implies ${\bf T}_{M_0}$ and $({\bf T}_{M_0})^{-1}$ are bounded on the norm $\|\cdot\|_{\dot{B}_{p,cw}^{\alpha,q}}$.
	
	Now we prove the claim \eqref{claim}. By \eqref{CW reproducing},
	\begin{equation}
		\begin{aligned}
			\left\|{\bf R}_{M_0} f\right\|_{\dot{B}_{p,cw}^{\alpha,q}}&=\left\{\sum_{k'\in\Z}\left\|2^{k'\alpha}\sum_{Q'\in Q_{cw}^{k'}} {\bf E}_{k'}{\bf R}_{M_0}(f)(x_{Q'})\chi_{Q'}(x)\right\|_{p}^{q}\right\}^{\frac{1}{q}}
			\\
			&=\left\{\sum_{k'\in\Z}\left\|2^{k'\alpha}\sum_{Q'\in Q_{cw}^{k'}}{\bf  E}_{k'}\left(\sum_{k\in\Z}\sum_{Q\in Q_{cw}^k}
			\omega(Q){R}_{M_0}\widetilde{E}_k(\cdot,x_Q){\bf E}_k(f)(x_Q)\right)(x_{Q'})\chi_{Q'}(x)\right\|_{p}^{q}\right\}^{\frac{1}{q}}
			\\
			&\leq\left\{\sum_{k'\in\Z}\left\|2^{k'\alpha}\sum_{Q'\in Q_{cw}^{k'}}\sum_{k\in\Z}\sum_{Q\in Q_{cw}^k}
			\omega(Q)\left|E_{k'}{R}_{M_0}\widetilde{E}_k(x_{Q'},x_Q){\bf E}_k(f)(x_Q)\right|\chi_{Q'}(x)\right\|_{p}^{q}\right\}^{\frac{1}{q}}.
		\end{aligned}
	\end{equation}
	By {\bf Lemma \ref{almost orthogonal estimate with dcz}}, choosing $\varepsilon_0$ such that $|\alpha|<\var_0$ and $\max\left\{{N\over{N+\var_0}},{N\over {N+\var_0+\alpha}}\right\}<p$, we have
	$$
	\left|{ E}_{k'}R_{M_0}\widetilde{E}_k(x_{Q'},x_Q)\right|\leq C\left\|{\bf R}_{M_0}\right\|_{dcz}2^{-|k-k'|\var_0}\frac{1}{V(x_{Q'},x_{Q},2^{-k\lor{-k'}}+d(x_{Q'},x_Q))}\left(\frac{2^{-k\lor{-k'}}}{2^{-k\lor{-k'}}+d(x_{Q'},x_Q)}\right)^{\var_0}.
	$$
	%we can write
	%	$$
	%	\left|E_{k'}R_{M_0}\widetilde{E}_k(x_{Q'},x_Q)\right|= C_{\var_0}\left\|R_{M_0}\right\|_{dcz}S_{k,k'}(x_{Q'},x_{Q}),
	%	$$
	%	where $S_{k,k'}(x_{Q'},x_Q)\leq2^{-|k-k'|\varepsilon_0}\frac{1}{V(x,y,2^{-k\lor{-k'}}+d(x_{Q'},x_Q))}\left(\frac{2^{-k\lor{-k'}}}{2^{-k\lor{-k'}}+d(x_{Q'},x_Q)}\right)^{\varepsilon_0}$.
	We choose $\theta_0$ such that $\max\left\{\frac{N}{N+\varepsilon_0},\frac{N}{N+\varepsilon_0+\alpha}\right\}<\theta_0\leq 1$ and $\theta_0<p$, then by {\bf Lemma \ref{iterate}}, since $Q_{cw}^k$ are collections of dyadic cubes with the  side length $2^{-k-M_1}$, there exists a constant $C=C(N,\alpha,\var_0,\theta_0,{p\over\theta_0},M_1)$ such that
	\begin{equation}
		\begin{aligned}
			\left\|{\bf R}_{M_0}f\right\|_{\dot{B}_{p,cw}^{\alpha,q}} &\leq C\left\|{\bf R}_{M_0}\right\|_{dcz}\left\{\left\|2^{k\alpha}\sum_{Q\in Q_{cw}^{k}} {\bf E}_{k}(f)(x_{Q})\chi_Q(x)\right\|_{p}^{q}\right\}^{\frac{1}{q}}
			\\
			&=C\left\|{\bf {\bf R}}_{M_0}\right\|_{dcz}\left\|f\right\|_{\dot{B}_{p,cw}^{\alpha,q}}.
		\end{aligned}
	\end{equation}
	Choosing $M_0$ such that $2^{M_0\delta}>2C$, we finish the proof of ${\bf Step ~1}$.	Now we turn to prove $\|f\|_{\dot{B}_{p,d}^{\alpha,q}}\sim \|f\|_{\dot{B}_{p,cw}^{\alpha,q}}$. We will do this by two steps.

	{\bf{Step 2:}} We prove
	~$	\left\|f\right\| _{\dot{B}_{p,d}^{\alpha,q}}~\leq~	C\left\|f\right\| _{\dot{B}_{p,cw}^{\alpha,q}}$.
	One can check that the Dunkl-Besov norm $\|f\|_{\dot{B}_{p,d}^{\alpha,q}}$ can be written as
	\begin{equation}\label{CW norm}
		\|f\|_{\dot{B}_{p,d}^{\alpha,q}}=\left\{\sum_{k\in\Z}\left\|2^{k\alpha}\sum_{Q\in Q^k_d}{\bf D}_k(f)(x_Q)\chi_Q(x) \right\|_p^q \right\}^{\frac{1}{q}}.
	\end{equation}
	To show the above estimate, by \eqref{CW reproducing}, we have
	\begin{equation}\label{hh}
	\begin{aligned}
		\left\|f\right\| _{\dot{B}_{p,d}^{\alpha,q}}
		&=\left\{\sum_{k'\in\Z}    \left\| 2^{k'\alpha } \sum_{Q'\in Q_{d}^{k'}}{{\bf D}_{k'} (f)(x_{Q'}) \chi{_{Q'}}(x) }\right\|_{p}^{q}\right\} ^{1\over q}
		\\
		&=\left\{\sum_{k'\in\Z}  \left\|2^{k'\alpha }      \sum_{Q'\in Q_d^{k'}}{{\bf D}_{k'}   \left(    \sum_{k\in\Z}\sum_{Q\in Q_{cw}^{k}} \omega(Q)\Tilde{E}_k(\cdot,x_{Q}){\bf E}_k(f)(x_{Q})    \right)    (x_{Q'}) \chi{_{Q'}}(x) }\right\|_{p}^{q}\right\} ^{1\over q}
		\\ 
		&\leq~\left\{\sum_{k'\in\Z}     \left\|{2^{k'\alpha }    \sum_{k\in\Z}\sum_{Q'\in Q_d^{k'}}  \sum_{Q\in Q_{cw}^k} \omega(Q)\left|D_{k'}  \Tilde{E}_k(x_{Q'},x_{Q}){\bf E}_k(f)(x_{Q}) \right|    \chi{_{Q'}}(x) }\right\|_{p}^{q}\right\} ^{1\over q}.
	\end{aligned}
	\end{equation}
	By   \cite[{\bf Lemma {2.11}}]{Hp}, we have
	\bel{ek}
	\left|D_{k'}\Tilde{E}_k(x_{Q'},x_{Q})\right|\leq C2^{-|k-k'|\varepsilon_0}\frac{1}{V(x_{Q'},x_Q,2^{-k\lor{-k'}}+d(x_{Q'},x_{Q}))}\left(\frac{2^{-k\lor{-k'}}}{2^{-k\lor{-k'}}+d(x_{Q'},x_{Q})}\right)^{\var_0}.
	\eeq
	Then by {\bf Lemma \ref{iterate}},
	\begin{equation}\label{HH}
	\begin{aligned}
		\left\{\sum_{k'\in\Z}      \left\|{2^{k'\alpha }   \sum_{k\in\Z}\sum_{Q'\in Q_d^{k'}}  \sum_{Q\in Q_{cw}^k} \omega(Q)\left|D_{k' } \Tilde{E}_k(x_{Q'},x_{Q}){\bf E}_k(f)(x_{Q}) \right|    \chi{_{Q'}}(x) }\right\|_{p}^{q}\right\} ^{1\over q}
		\\
		~\leq~C~ \left\{\sum_{k\in\Z}\left\|2^{k\alpha }\sum_{Q\in Q_{cw}^k}{\bf E}_k(f)(x_Q)\chi_Q(x)\right\|_p^{q} \right\}^{1\over q}.
	\end{aligned}
	\end{equation}
	Therefore,  $\left\|f\right\| _{\dot{B}_{p,d}^{\alpha,q}}\leq C\left\|f\right\| _{\dot{B}_{p,cw}^{\alpha,q}}$.

	{\bf {Step 3:}} We prove
	$	\left\|f\right\| _{\dot{B}_{p,cw}^{\alpha,q}}~\leq~	C\left\|f\right\| _{\dot{B}_{p,d}^{\alpha,q}}$.
	By  {\bf {Step~1}}, we  have
	\begin{equation}
		\left\|f\right\|_{\dot{B}_{p,cw}^{\alpha,q}}\sim \left\|{\bf T}_{M_0}f\right\|_{\dot{B}_{p,cw}^{\alpha,q}}.
	\end{equation}
	%\bel{3R1Rm}
	%\left\|f\right\|_{\dot{\B}_{p,cw}^{\alpha,q}}\leq \left\|T_{M_0}f\right\|_{\dot{\B}_{p,cw}^{\alpha,q}}+\left\|(R_{M_0})f\right\|_{\dot{\B}_{p,cw}^{\alpha,q}}\leq \left\|T_{M_0}f\right\|_{\dot{\B}_{p,cw}^{\alpha,q}}+{1\over 2}\left\|f\right\|_{\dot{\B}_{p,cw}^{\alpha,q}}.
	%\eeq
Then we only need to show
	$$
	\left\|{\bf T}_{M_0}f\right\|_{\dot{B}_{p,cw}^{\alpha,q}}\leq C\left\|f\right\|_{\dot{B}_{p,d}^{\alpha,q}}.
	$$
		By the definition of ${\bf T}_{M_0}f$ in \eqref{T_{M_0}}, we obtain that
	\begin{equation}
		\begin{aligned}
			\left\|{\bf T}_{M_0}f\right\| _{\dot{B}_{p,cw}^{\alpha,q}}
			&=\left\{\sum_{k'\in\Z}    \left\|2^{k'\alpha } \sum_{Q'\in Q_{cw}^{k'}}{{\bf E}_{k' }({\bf T}_{M_0}f)(x_{Q'}) \chi{_{Q'}}(x) }\right\|_{p}^{q}\right\} ^{1\over q}
			\\
			&=\left\{\sum_{k'\in\Z}    \left\|2^{k'\alpha }    \sum_{Q'\in Q_{cw}^{k'}}{{\bf E}_{k'}  \left(    \sum_{k\in\Z}\sum_{Q\in Q_d^{k}} \omega(Q)D^{M_0}_k(\cdot,x_{Q}){\bf D}_k(f)(x_{Q})    \right)    (x_{Q'}) \chi{_{Q'}}(x) }\right\|_{p}^{q}\right\} ^{1\over q}
			\\
			&\leq~C~\left\{\sum_{k'\in\Z}       \left\|{2^{k'\alpha}  \sum_{k\in\Z}\sum_{Q'\in Q^{k'}_{cw}}  \sum_{Q\in Q^k_d} \omega(Q)\left|E_{k'}   D^{M_0}_k(x_{Q'},x_{Q}){\bf D}_k(f)(x_{Q}) \right|    \chi{_{Q'}}(x) }\right\|_{p}^{q}\right\} ^{1\over q}.
		\end{aligned}
	\end{equation}
	By  \cite[{\bf Lemma 2.11}]{Hp}, we get
	\begin{equation}\label{EkDk}
	\left|E_{k'}D^{M_0}_{k}(x_{Q'},x_Q)\right|\leq C~2^{-|k-k'|\varepsilon_0}\frac{1}{V(x_{Q'},x_Q,2^{-k\lor{-k'}}+d(x_{Q'},x_Q))}\left(\frac{2^{-k\lor{-k'}}}{2^{-k\lor{-k'}}+d(x_{Q'},x_Q)}\right)^{\varepsilon_0}.
	\end{equation}
	Then by {\bf Lemma \ref{iterate}}, we have
	\begin{equation}\label{ekdk}
	\begin{aligned}
		\left\{\sum_{k'\in\Z}      \left\|{2^{k'\alpha }   \sum_{k\in\Z}\sum_{Q'\in Q_{cw}^{k'}}  \sum_{Q\in Q_d^k} \omega(Q)\left|E_{k '}  D^{M_0}_k(x_{Q'},x_{Q}){\bf D}_k(f)(x_{Q}) \right|    \chi{_{Q'}}(x) }\right\|_{p}^{q}\right\} ^{1\over q}
		\\ 
		~\leq~ C~\left\{\sum_{k\in\Z}\left\|2^{k\alpha }\sum_{Q\in Q_d^k}{\bf D}_k(f)(x_Q)\chi_Q(x) \right\|_p^q\right\}^{1\over q}=C~\left\|f\right\|_{\dot{B}_{p,d}^{\alpha,q}}.
	\end{aligned}
	\end{equation}
\end{proof}
	
\begin{proof}[{Proof of Theorem  \ref{Dunkl reproducing} }]
 By applying {\bf   Lemma  \ref{Tbounded}} ,  we have ${\bf T}_{M_0}$ and $	({\bf T}_{M_0})^{-1}$ are bounded on   $L^2\cap \dot{B}_{p,d}^{\alpha,q}$. 	Set $h:=({\bf T}_{M_0})^{-1} f$, then we obtain
	$$
	f(x)={\bf T}_{M_0} h(x)=\sum_{k\in\Z}\sum_{Q\in Q_d^k}\omega(Q)D^{M_0}_k(x,x_Q){\bf D}_k(h)(x_Q),
	$$
	where   $\|f\|_{2}\sim\|h\|_{2}$,   $\|f\|_{\dot{B}_{p,d}^{\alpha,q}}\sim\|h\|_{\dot{B}_{p,d}^{\alpha,q}}$.
	To prove the above series converges in $L^2\cap \dot{B}_{p,d}^{\alpha,q}$, we just need to show
	\begin{equation}\label{jk}
	\lim_{m\to\infty}	\left\{\sum_{k'\in \Z}\left\|2^{k'\alpha }\sum_{Q'\in Q_d^{k'}}\sum_{|k|>m}\sum_{Q\in Q_d^{k}} \left|D_{k'}D_k^{M_0}(x_{Q'},x_Q){\bf D}_k(h)(x_Q)\right|\chi_{Q'}(x)\right\|_p^q\right\}^{1\over q}=0.
	\end{equation}
	By  \cite[{\bf Lemma 2.11}] {Hp},
	\begin{equation}\label{DkDj}
	\left|D_{k'}D_k^{M_0}(x_{Q'},x_Q)\right|
	~\leq~2^{-|k-k'|\var_0}\frac{1}{V(x_{Q'},x_Q,2^{-k\lor{-k'}}+d(x_{Q'},x_Q))}\left(\frac{2^{-k\lor{-k'}}}{2^{-k\lor{-k'}}+d(x_{Q'},x_Q)}\right)^{\varepsilon_0}.
	\end{equation}
	Then by {\bf Lemma \ref{iterate}},
	\begin{equation}\label{result2}
	\begin{aligned}
		\left\{\sum_{k'\in \Z}\left\|2^{k'\alpha }\sum_{Q'\in Q^{k'}_d}\sum_{|k|>m}\sum_{Q\in Q^{k}_d} \left|D_{k'}D_k^{M_0}(x_{Q'},x_Q){\bf D}_k(h)(x_Q)\right|\chi_{Q'}(x)\right\|_p^q\right\}^{1\over q}
		\\
		\leq ~\left\{\sum_{|k|>m}\left\|2^{k\alpha } \sum_{Q\in Q_d^k}{\bf D}_k(h)(x_Q)\chi_Q(x)\right\|^{q}_p\right\}^{1\over q}.
	\end{aligned}
	\end{equation}
	Since $\left\|h\right\|_{\dot{B}_{p,d}^{\alpha,q}}\sim \left\|f\right\|_{\dot{B}_{p,d}^{\alpha,q}}$, the last term tends to $0$ as $m\to \infty$. Then the proof of {\bf Theorem \ref{Dunkl reproducing}}  is complete.
	\end{proof}
	
	\section{Duality estimate}
	\setcounter{equation}{0}
	
	In  {\bf Definition \ref{index}}, we define the dual indexes  $\alpha',p',q'$ of $\alpha,p,q$.
	As mentioned, the starting point  of the Besov spaces in the Dunkl setting  is the duality estimate {\bf Theorem \ref{dual prop}}.
	
    \begin{proof}[Proof of Theorem \ref{dual prop}.] 
    	Applying  {\bf Theorem \ref{Dunkl reproducing}} for $f\in L^2\cap\dot{B}_{p,d}^{\alpha,q}$,
	there exists a function $h\in L^2(d\omega)$ with $\|h\|_2\sim\|f\|_2$ and $\|h\|_{ \dot{B}_{p,d}^{\alpha,q}}\sim \|f\|_{ \dot{B}_{p,d}^{\alpha,q}}$ such that
	\begin{equation}
		\langle f,g\rangle
		=\sum_{k\in\Z}\sum_{Q\in Q_d^k} \omega(Q){\bf D}_k(h)(x_Q) {\bf D}_k^{M_0}(g)(x_Q).
	\end{equation}
	Now we divide the proof into several cases for different ranges of $p,~q$.

	{\bf Case 1.}\qquad Suppose  $~ 1<p,~ q<\infty$. Then by H\"{o}lder inequality,
	\begin{equation}\label{pq1}
	\begin{aligned}
		\left|\langle f,g\rangle\right|
		&\leq~ \sum_{k\in\Z}\left\{\sum_{Q\in Q_d^k} \omega(Q)  \left\{2^{k\alpha}\left|  {\bf D}_k(h)(x_Q)    \right|\right\}^p \right\}^{1\over p}  \left\{\sum_{Q\in Q_d^k} \omega(Q)  \left\{2^{-k\alpha}\left|{\bf D}_k^{M_0}(g)(x_Q)\right|\right\}^{p'}\right\}^{1\over p'}
		\\
		&\leq~\left\{\sum_{k\in\Z}\left\{\sum_{Q\in Q_d^k} \omega(Q)  \left\{2^{k\alpha}\left|  {\bf D}_k(h)(x_Q)    \right|\right\}^p \right\}^{q\over p}  \right\}^{1\over q}     \left\{\sum_{k \in\Z}\left\{\sum_{Q\in Q_d^k} \omega(Q)  \left\{2^{-k\alpha}\left|{\bf D}_k^{M_0}(g)(x_Q)\right|\right\}^{p'} \right\}^{q'\over p'} \right\} ^{1\over q'}.
	\end{aligned}
	\end{equation}
	For the second term,  similar to the {\bf step 2} in {\bf Lemma \ref{Tbounded}}, we get 
	\begin{equation}\label{2claim}
		\left\{\sum_{k \in\Z}\left\{\sum_{Q\in Q_d^k} \omega(Q)  \left\{2^{-k\alpha}\left|{\bf D}_k^{M_0}(g)(x_Q)\right|\right\}^{p'} \right\}^{q'\over p'} \right\} ^{1\over q'}
		\leq C~ \|g\|_{\dot{B}_{p',d}^{-\alpha,q'}}.
	\end{equation}
	Note that in this case, $\alpha'=-\alpha$. Then we have
	\begin{equation}
		\left|\langle f,g\rangle\right|
		~\leq~ C  \left\|h\right\|_{\dot{B}_{p,d}^{\alpha,q}}\left\|g\right\|_{\dot{B}_{p',d}^{\alpha',q'}}
		~=~ C\left\|f\right\| _{\dot{B}_{p,d}^{\alpha,q}}\left\|g\right\|_{\dot{B}_{p',d}^{\alpha',q'}}.
	\end{equation}

	{\bf Case 2.}\qquad  Suppose   $~1<p<\infty,~0<q\leq 1$. Then by H\"{o}lder inequality,
	\begin{equation}\label{pq2}
	\begin{aligned}
		\left|\langle f,g\rangle\right|
		&\leq~ \sum_{k\in\Z}\left\{\sum_{Q\in Q_d^k} \omega(Q)  \left\{2^{k\alpha}\left|  {\bf D}_k(h)(x_Q)  \right|\right\}^p \right\}^{1\over p}
		\left\{\sum_{Q\in Q_d^k} \omega(Q)  \left\{2^{-k\alpha}\left|{\bf D}_k^{M_0}(g)(x_Q)\right|\right\}^{p'}\right\}^{1\over p'}
		\\
		&\leq~C~\left\{\sum_{k\in\Z}\left\{\sum_{Q\in Q_d^k} \omega(Q)\left\{2^{k\alpha}\left|  {\bf D}_k(h)(x_Q) \right|\right\}^p \right\}^{q\over p} \right\}^{1\over q}  	\left\|g\right\| _{\dot{B}_{p'}^{-\alpha,\infty}}
		\\\\ \ds
		&= C \left\|h\right\| _{\dot{B}_{p,d}^{\alpha,q}}    \left\|g\right\| _{\dot{B}_{p',d}^{-\alpha,\infty}}  =C \left\|f\right\| _{\dot{B}_{p,d}^{\alpha,q}}    \left\|g\right\| _{\dot{B}_{p',d}^{\alpha',\infty}}.
	\end{aligned}
	\end{equation}

	{\bf Case 3.}\qquad  For  $\max\left\{{N\over {N+1}},~{N\over{N+1+\alpha}}\right\}<p\leq1,~1<q<\infty$, we have
	\begin{equation}\label{pq3}
		\begin{aligned}
			\left|\langle f,g\rangle\right|
			&\leq \sum_{k\in\Z} \sum_{Q\in Q_d^k} \omega(Q)  2^{k[\alpha-N({1\over p} -1)]}\left|  {\bf D}_k(h)(x_Q)    \right| ~  \sup_{{Q\in Q_d^k} }\left\{ 2^{-k[\alpha-N({1\over p} -1)]}\left|{\bf D}_k^{M_0}(g)(x_Q)\right|\right\}
			\\
			&\leq \left\{\sum_{k\in\Z} \left\{\sum_{Q\in Q_d^k} \omega(Q)  2^{k[\alpha-N({1\over p} -1)]}\left|  {\bf D}_k(h)(x_Q)    \right| \right\}^q\right\}^{1\over q}~   \left\{\sum_{k\in\Z}\left\{\sup_{{Q\in Q_d^k}}2^{-k[\alpha-N({1\over p} -1)]}\left|{\bf D}_k^{M_0}(g)(x_Q)\right|\right\}^{q'}\right\}^{1\over {q'}}.
		\end{aligned}
	\end{equation}
	Then  we have
	\begin{equation}
		\left|\langle f,g\rangle\right|\leq C \left\|h\right\| _{\dot{B}_{1,d}^{\alpha-{N({1\over p}-1)},q}}\left\|g\right\| _{\dot{B}_{\infty,d}^{-\alpha+{N({1\over p}-1)},q'}}= C \left\|f\right\| _{\dot{B}_{1,d}^{\alpha-{N({1\over p}-1)},q}}\left\|g\right\| _{\dot{B}_{\infty,d}^{\alpha',q'}}.
	\end{equation}
It remains to prove $\left\|f\right\| _{\dot{B}_{1,d}^{\alpha-{N({1\over p}-1)},q}}\leq C \left\|f\right\| _{\dot{B}_{p,d}^{\alpha,q}}$. 		Since  $p\leq1$, we get
		\begin{equation}
	\begin{aligned}
		\left\|f\right\| _{\dot{B}_{1,d}^{\alpha-{N({1\over p}-1)},q}}  &=
		\left\{\sum_{k\in\Z} ~2^{k[\alpha-N({1\over p} -1)]q}\left\{\sum_{Q\in Q_d^k} \omega(Q)  \left|  {\bf D}_k(f)(x_Q)    \right| \right\}^q\right\}^{1\over q}
		\\
		&\leq \left\{\sum_{k\in\Z}~2^{k[\alpha-N({1\over p} -1)]q} \left\{\sum_{Q\in Q_d^k} {\omega(Q)}^{1-{1\over p}}  ~{\omega(Q)}^{1\over p}~\left|  {\bf D}_k(f)(x_Q) \right| \right\}^{q}\right\}^{1\over q}
		\\ 
		&\leq\left\{\sum_{k\in\Z}~2^{k[\alpha-N({1\over p} -1)]q} \left\{\sum_{Q\in Q_d^k} {\omega(Q)}^{p-{1} } ~{\omega(Q)}~\left|  {\bf D}_k(f)(x_Q) \right|^p \right\}^{q\over p}\right\}^{1\over q}.
\end{aligned}
\end{equation}

	By \eqref{measure of ball}, for  $Q\in Q_d^k$, we have $\omega(Q)\geq~C~2^{-kN}.$ Hence,
	\begin{equation}\label{embedding result}
	\begin{aligned}
		\left\|f\right\| _{\dot{B}_{1,d}^{\alpha-{N({1\over p}-1)},q}}  
	&\leq C \left\{\sum_{k\in\Z}~2^{k[\alpha-N({1\over p} -1)]q} \left\{\sum_{Q\in Q_d^k} 2^{-kN(p-1)} ~{\omega(Q)}~\left|  {\bf D}_k(f)(x_Q) \right|^p \right\}^{q\over p}\right\}^{1\over q}
		\\
	&\leq C\left\{\sum_{k\in\Z} ~2^{k\alpha q}\left\{\sum_{Q\in Q_d^k} {\omega(Q)}~\left|  {\bf D}_k(f)(x_Q) \right|^p \right\}^{q\over p}\right\}^{1\over q} =C\left\|f\right\| _{\dot{B}_{p,d}^{\alpha,q}}.
	\end{aligned}
	\end{equation}
    Thus,
	\begin{equation}
	\begin{aligned}
		\left|\langle f,g\rangle\right|~\leq~ \left\|f\right\| _{\dot{B}_{p,d}^{\alpha,q}}    \left\|g\right\| _{\dot{B}_{\infty,d}^{\alpha',q'}}.
	\end{aligned}
	\end{equation}

	{\bf Case 4.}\qquad For  $\max\left\{{N\over {N+1}},~{N\over{N+1+\alpha}}\right\}<p\leq1,~0<q\leq1$, we have
	\begin{equation}\label{pq4}
		\left|\langle f,g\rangle\right|
		\leq~ \sum_{k\in\Z} \sum_{Q\in Q_d^k} \omega(Q)   2^{k[\alpha-N({1\over p} -1)]}\left| { \bf D}_k(h)(x_Q)\right|~  \left\{\sup_{k\in \Z,Q\in Q_d^k} 2^{-k[\alpha-N({1\over p} -1)]}\left|{\bf D}_k^{M_0}(g)(x_Q)\right|\right\}.
	\end{equation}
	Since $q\leq 1$, we have
	\begin{equation}\label{xjq}
	\begin{aligned}
		\left|\langle f,g\rangle\right|&\leq~ \left\{\sum_{k\in\Z} \left\{\sum_{Q\in Q_d^k} \omega(Q)   2^{k[\alpha-N({1\over p} -1)]}\left|  {\bf D}_k(h)(x_Q)    \right| \right\}^q\right\}^{1\over q}~      \left\|g\right\| _{\dot{B}_{\infty}^{-\alpha+{N({1\over p}-1)},\infty}}
		\\
		&= C\left\|f\right\| _{\dot{B}_{1,d}^{\alpha-{N({1\over p}-1)},q}}    \left\|g\right\| _{\dot{B}_{\infty,d}^{\alpha',\infty}}.
	\end{aligned}
	\end{equation}
	From \eqref{embedding result},  $\left\|f\right\| _{\dot{B}_{1,d}^{\alpha-{N({1\over p}-1)},q}}\leq C\left\|f\right\| _{\dot{B}_{p,d}^{\alpha,q}}$. Then we obtain that
	\begin{equation}
		\left|\langle f,g\rangle\right|~\leq~ \left\|f\right\| _{\dot{B}_{p,d}^{\alpha,q}}    \left\|g\right\| _{\dot{B}_{\infty,d}^{\alpha',\infty}}.
    \end{equation}
\end{proof}
	
	\section {Dunkl-Besov space}
	\setcounter{equation}{0}
	In this section, we will prove that Dunkl-Besov space $\dot{\B}_{p,d}^{\alpha,q}$ is the closure of $L^2\cap \dot{B}_{p,d}^{\alpha,q}$, so $\dot{\B}_{p,d}^{\alpha,q}$ is a complete space. Before introducing the Dunkl-Besov space,  we explain how to extend the operator ${\bf D_k}$ to the distribution spaces.  Because $\max\{{N\over N+1},{N\over {N+1+\alpha}}\}<p<\infty$, $0<q<\infty$ and $|\alpha|<1$, then  by {\bf Lemma} \ref{basis}, $D_j{(x, \cdot)} $ and $D_j{(\cdot, y)} $ belong to $L^2\cap\dot{B}^{\alpha,q}_{p,d}$ for any fixed $j\in \Z$  and $x,y\in \R^n$. Then for $f\in \left(L^2\cap \dot{B}_{p,d}^{\alpha,q}\right)'$,  the operator $\D_j(f)(y)= \langle f,D_j(\cdot,y)\rangle$ is well defined. Now we establish  a weak-type discrete  Calder\'{o}n reproducing formula in the distribution sense:
		
\begin{proof}[Proof of Theorem \ref{distribution converge}.] For $(i)$, suppose that  $\{f_n\}_{n=1}^{\infty}$ is a sequence in $L^2(d\omega)$ with  $ \| f_n-f_m\|_{\dot{B}_{p,d}^{\alpha,q}}\to 0$ as $n,m\to \infty$, then  for each $g\in L^2(d\omega)$ with $\|g\|_{\dot{B}_{p',d}^{\alpha',q'}}<\infty$,  by {\bf Theorem \ref{dual prop}},  we have $\lim\limits_{n,m\rightarrow\infty}\langle f_n-f_m, g\rangle=0$.  Therefore, there exists $f$, as  a distribution on $(L^2\cap\dot{B}^{\alpha',q'}_{p',d})'$, such that for each $g\in L^2(d\omega)$ with $\|g\|_{\dot{B}_{p',d}^{\alpha',q'}}<\infty$,
		$$
		\langle f,g\rangle=\lim\limits_{n\rightarrow\infty}\langle f_n, g\rangle.
		$$
		
		To show $(ii)$, since
		$$
		\left\|f-f_n\right\|_{\dot{B}_{p,d}^{\alpha,q}}=\left\{\sum_{k\in\Z}\left(\sum_{Q\in Q_d^k}\omega(Q)\left|{\bf D}_k(f-f_n)(x_Q)\right|^p\right)^{q\over p}\right\}^{1\over q},
		$$
		and ${\bf D}_k(f-f_n)(x_Q)=\lim\limits_{m\rightarrow\infty} {\bf D}_k(f_m-f_n)(x_Q)$. Then by Fatou's lemma,
		\begin{equation}
			\begin{aligned}
				\|f-f_n\|_{{\dot{B}_{p,d}^{\alpha,q}}}&\leq\liminf\limits_{m\rightarrow\infty} \left\{\sum_{k\in\Z}\left(\sum_{Q\in Q_d^k}\omega(Q)\left|{\bf D}_k(f_m-f_n)(x_Q)\right|^p\right)^{q\over p}\right\}^{1\over q}
				\\
				&=\liminf\limits_{m\rightarrow\infty}\|f_m-f_n\|_{\dot{B}_{p,d}^{\alpha,q}}\to 0, \qquad\textit{as $n\to\infty$}.
			\end{aligned}
		\end{equation}
		Hence, $\|f\|_{\dot{B}^{\alpha,q}_{p,d}}=\lim\limits_{n\rightarrow\infty}\|f_n\|_{\dot{B}^{\alpha,q}_{p,d}}$.
		
		To show $(iii)$, denote $h_n=({\bf T}_{M_0})^{-1}(f_n)$. By {\bf Lemma \ref{Tbounded}}, $\{h_n\}$ is a Cauchy sequence in $L^2\cap \dot{B}_{p,d}^{\alpha,q}$. Then by the same argument with $(i)$ and $(ii)$ , there exists  a distribution $h\in \left(L^2\cap\dot{B}^{\alpha',q'}_{p',d}\right)'$, such that for every $g\in L^2\cap\dot{B}^{\alpha',q'}_{p',d}$,
		$$
		\langle h,g\rangle=\lim\limits_{n\rightarrow\infty}\langle h_n, g\rangle
		$$
		and  $\|h\|_{{\dot{B}_{p,d}^{\alpha,q}} }~=\lim\limits_{n\to\infty} \|h_n\|_{{\dot{B}_{p,d}^{\alpha,q}} } \sim \lim\limits_{n\to\infty} \|f_n\|_{{\dot{B}_{p,d}^{\alpha,q}} } ~=\|f\|_{{\dot{B}_{p,d}^{\alpha,q}} }$.
		Applying the proof of  {\bf Theorem \ref{dual prop}}, we have
		\begin{equation}\label{iii1}
		\left|\sum\limits_{k\in \Z}\sum_{Q\in Q_d^k} \omega(Q){\bf D}_k(h)(x_Q) {\bf D}_k^{M_0}(g)(x_Q)\right|~\leq~C ~\left\|f\right\| _{\dot{B}_{p,d}^{\alpha,q}}    \left\|g\right\| _{\dot{B}_{p',d}^{\alpha',q'}},
		\end{equation}
		which implies that the series $ \sum\limits_{k\in \Z}\sum\limits_{Q\in Q_d^k} \omega(Q){\bf D}_k(h)(x_Q) {\bf D}_k^{M_0}(g)(x_Q)$  is a distribution in  $\left(L^2\cap\dot{B}^{\alpha',q'}_{p',d}\right)'$. Moreover, by the weak-type discrete  Calder\'{o}n reproducing formula of $f_n$ in  {\bf Theorem  \ref{Dunkl reproducing}} for each $g\in L^2\cap\dot{B}^{\alpha',q'}_{p',d}$,
		\begin{equation}\label{limit series}
		\langle f,g\rangle=\lim\limits_{n\rightarrow\infty}\langle f_n, g\rangle=\lim\limits_{n\rightarrow\infty}\left\langle\sum\limits_{k\in\Z}\sum_{Q\in Q_d^k} \omega(Q)D_k^{M_0}(\cdot, x_Q){\bf D}_k(h_n)(x_Q) , g\right\rangle,
		\end{equation}
		where $\|f_n\|_2 \sim\|h_n\|_2$  and $\|f_n\|_{\dot{B}^{\alpha,q}_{p,d}}\sim \|h_n\|_{\dot{B}^{\alpha,q}_{p,d}}$.
		Using the similar way as in the proof of {\bf Theorem \ref{dual prop}}, we obtain 
		\begin{equation}\label{h-hn}
		\left|\left\langle\sum\limits_{k\in\Z}\sum_{Q\in Q_d^k} \omega(Q)D_k^{M_0}(\cdot, x_Q){\bf D}_k(h-h_n)(x_Q) , g\right\rangle\right| \leq C \left\|h-h_n\right\| _{\dot{B}_{p,d}^{\alpha,q}}    \left\|g\right\| _{\dot{B}_{p',d}^{\alpha',q'}},
		\end{equation}
		where the last term above tends to zero as $n\to \infty$ and hence,
		\begin{equation}\label{fntof}
		\langle f,g\rangle=\lim\limits_{n\rightarrow\infty}\langle f_n, g\rangle=\left\langle\sum_{k\in\Z}\sum_{Q\in Q_d^k} \omega(Q)D_k^{M_0}(\cdot, x_Q){\bf D}_k(h)(x_Q) , g\right\rangle.
		\end{equation}
\end{proof}

		{\bf{Theorem \ref {distribution converge}}} indicates  that one can consider $L^2\cap\dot{B}_{p',d}^{\alpha',q'}$, the subspace of $f\in L^2(d\omega)$ with  the norm $\|f\|_{\dot{B}_{p',d}^{\alpha',q'}}<\infty$,  as  the test function space and  $(L^2\cap{\dot{B}_{p',d}^{\alpha',q'}})'$, as the distribution space. The Dunkl-Besov space is defined by {\bf Definition \ref{inf}}. We remark that in the {\bf Definition \ref{inf}}, the series $
		\sum_{k\in\Z}\sum_{Q\in Q_d^k}\omega(Q)D_k^{M_0}(x,x_Q)\lambda_Q,$  with $\left\{\sum_{k\in\Z} 2^{k\alpha q} \left\{\sum_{Q\in Q_d^k}\omega_Q~{\left|\lambda_{Q}\right| ^p}\right\}^{q\over p} \right\} ^{1\over q}<\infty$ defines a distribution  in $(L^2\cap{\dot{B}_{p',d}^{\alpha',q'}})'$.  Indeed, applying the proof of {\bf{Theorem \ref {dual prop}}}, for each $g\in L^2\cap{\dot{B}_{p',d}^{\alpha',q'}}$,
		$$
		\left| \sum\limits_{k\in Z}\sum_{Q\in Q_d^k}\omega(Q)\lambda_Q {\bf D}_k^{M_0}(g)(x_Q)\right|~\leq~C ~ \left\{\sum_{k\in\Z} 2^{k\alpha q} \left\{\sum_{Q\in Q_d^k}\omega_Q~{\left|\lambda_{Q}\right| ^p}\right\}^{q\over p} \right\} ^{1\over q} \left\|g\right\| _{\dot{B}_{p',d}^{\alpha',q'}}.
		$$

		We now show
		$\dot{\B}^{\alpha,q}_{p,d}(d\omega)=\bar{L^2\cap\dot{B}^{\alpha,q}_{p,d}}$.
		
\begin{proof}[Proof of Theorem \ref{completness}.]  Suppose $f\in \dot{\B}^{\alpha,q}_{p,d}(d\omega)$,  from  {\bf Definition \ref{inf}},  there exists a  sequence  $\lambda(Q)_{k\in \Z,Q\in Q_d^k}$  such that
		$$
		f=\sum_{k\in\Z}\sum_{Q\in Q_d^k}\omega(Q)D_k^{M_0}(\cdot,x_Q)\lambda_Q,
		$$
		with $\left\{\sum_{k} 2^{k\alpha q} \left\{\sum_{Q\in Q_d^k}\omega (Q){\left|\lambda_{Q}\right| ^p}\right\}^{q\over p} \right\} ^{1\over q}<\infty.$
		Let 	
		\begin{equation}\label{fn reproducing}
		f_n(x)=\sum_{|k|\leq n}\sum_{{Q\in Q_d^k},{Q\subset B(0,n)}} \omega(Q)D_k^{M_0}(x,x_Q)\lambda_{Q}.
		\end{equation}
		We observe that $f_n\in L^2(d\omega)$,  then by \eqref{fn reproducing}, we have
		\begin{equation}\label{fn Besov}
		\begin{aligned}
			\left\|f_n\right\| _{\dot{B}_{p,d}^{\alpha,q}}  &=\left\{\sum_{k'\in \Z} \left\|2^{k'\alpha }\sum_{Q'\in Q_d^{k'}}{\bf D}_{k'} (f_n)(x_{Q'})\chi_{Q'}(x)\right\|_p^{q}          \right\} ^{1\over q}
			\\
			 &=\left\{\sum_{k'\in \Z} \left\|2^{k'\alpha } \sum_{Q'\in Q_d^{k'}}  \sum_{|k|\leq n}   {\sum_{{Q\in Q_d^k},{Q\subset B(0,n)}} \omega(Q) D_{k'} D_k^{M_0}(x_{Q'},x_Q) }   \lambda_{Q} \chi_{Q'}(x)\right\|_p^{q}  \right\} ^{1\over q}.
		\end{aligned}
		\end{equation}
		Then by  {\bf	Lemma \ref{iterate}}, we obtain that
		\begin{equation}\label{fn Besov norm}
		\begin{aligned}
			\left\|f_n\right\| _{\dot{B}_{p,d}^{\alpha,q}} \leq C \left\{\sum_{k\in\Z}\left\|2^{k\alpha}\sum_{Q\in Q_d^k}\lambda_Q\chi_Q(x) \right\|_p^q \right\}^{\frac{1}{q}}= \left\{\sum_{k\in \Z} 2^{k\alpha q} \left\{\sum_{Q\in Q_d^k}\omega (Q){\left|\lambda_{Q}\right| ^p}\right\}^{q\over p} \right\} ^{1\over q}<\infty.
		\end{aligned}
		\end{equation}
		Thus, $f_n\in  L^2\cap\dot{B}^{\alpha,q}_{p,d}$ and $f_n$ converges to $f$ in the sense of distributions on $L^2\cap \dot{B}_{p',d}^{\alpha',q'}$. To see $f_n$ converges to $f$ in  $\dot{\B}^{\alpha,q}_{p,d}$ as $n$ tends to $\infty$, let $E_{n}^k=\left\{Q:~Q\in Q_d^{k},Q\subset B(0,n)\right\}$ and $E_{n,m}^{k,c}=E_n^k\setminus E^k_m$ with $n\geq m$,
		\begin{equation}\label{fn-fm1}
		\begin{aligned}
			\|f_n-f_m\|_{{\dot{B}_{p,d}^{\alpha,q}} }&=\left\{\sum_{k'\in \Z} 2^{k'\alpha q}     \left\{\sum_{Q'\in Q_d^{k'}}\omega (Q'){\left|{\bf D}_{k'}(f_{n}-f_{m})(x_{Q'})\right| ^p}\right\}^{q\over p} \right\} ^{1\over q}
			\\ 
		    &=\left\{\sum_{k'\in\Z}      \left\|2^{k'\alpha }\sum_{Q'\in Q_d^{k'}}{{\bf D}_{k'}(f_{n}-f_{m})(x_{Q'})\chi_{Q'}(x)}\right\|^{q}_{p} \right\} ^{1\over q}
			\\
			&\leq~\left\{\sum_{k'\in\Z}   \left\| 2^{k'\alpha } \sum_{Q'\in Q_d^{k'}} \sum_{k=m+1}^n\sum_{Q\in E_{n,m}^{k,c}} \omega(Q)  \left| D_{k'} D_{k}^{M_0}(\cdot, x_{Q})\lambda _{Q}\right| \chi_Q'(x)   \right\|^{q}_{p} \right\} ^{1\over q}
			\\
			&\leq~C\left\{\sum_{k=m+1}^n\left\|2^{k\alpha }\sum_{Q\in E_{n,m}^{k,c}}\lambda_Q\chi_Q(x) \right\|_p^q \right\}^{\frac{1}{q}} \to 0
		\end{aligned}
		\end{equation}
		as $n,m$ tend to $\infty$, where the last inequality follows from {\bf Lemma} \ref{iterate}  and hence, $f\in \bar{L^2\cap\dot{B}^{\alpha,q}_{p,d}}$.
		
		Conversely, if $f\in \bar{L^2\cap\dot{B}^{\alpha,q}_{p,d}}$ , by {\bf{Theorem \ref {distribution converge}}},  there exists ~$h\in \left(L^2\cap\dot{B}^{\alpha',q'}_{p',d}\right)'$  with $\|f\|_{\dot{B}^{\alpha,q}_p}\sim \|h\|_{\dot{B}^{\alpha,q}_{p,d}}$ such  that for each  $g\in L^2\cap\dot{B}^{\alpha',q'}_{p',d}$,
		
		\begin{equation}\label{fg1}
		\langle f,g\rangle=\left\langle\sum\limits_{k\in\Z}\sum\limits_{Q\in Q_d^k}\omega(Q)D_k^{M_0}(\cdot,x_Q){\bf D}_k(h)(x_Q), g\right\rangle.
		\end{equation}
	Set $\lambda_Q=D_k(h)(x_Q)$ with $Q\in Q_d^k$. We obtain a wavelet-type decomposition of ~$f\in  \left(L^2\cap\dot{B}^{\alpha',q'}_{p',d}\right)'$ in the distribution sense:
		$$
		f=\sum_{k\in\Z}\sum_{Q\in Q_d^k}\omega(Q)D_k^{M_0}(x,x_Q)\lambda_Q.
		$$ 
		Moreover, 
		\begin{equation}\label{converse}
		\|f\|_{\dot{\B}^{\alpha,q}_{p,d}}=\inf \left\{\sum_{k\in\Z} 2^{k\alpha q} \left\{\sum_{Q\in Q_d^k}\omega_Q~{\left|\lambda_{Q}\right| ^p}\right\}^{q\over p} \right\} ^{1\over q}~\leq~C~\|f\|_{\dot{B}^{\alpha,q}_{p,d}}.
		\end{equation}
		By definition,  we have $f\in \dot{\B}^{\alpha,q}_{p,d}(d\omega)$.
	\end{proof}

		In the end, we give the following remark.
		In {\bf Definition {\ref{inf}}}, we define the  norm of Dunkl-Besov space as  \eqref {finf}. However, from {\bf Theorem {\ref {completness}}}, for $ f\in \dot{\B}_{p,d}^{\alpha,q}$,  there exists a  Cauchy sequence in $L^2\cap\dot{B}^{\alpha,q}_{p,d}$ converges to  $f$,  then we have another definition of the norm of $f$ :
		$$
		\left\|f\right\| _{\dot{B}_{p,d}^{\alpha,q}}  =\left\{\sum_{k\in\Z} \left\|2^{k\alpha }\sum_{Q\in Q_d^k}{\bf D}_k (f)(x_Q)\chi_Q(x)\right\|_p^{q}          \right\} ^{1\over q}
		$$		
		We will show that the two  definitions are  equivalent. That is
		$$
		\inf_{\{\lambda_Q\}_{k\in \Z,Q\in Q_d^k}}\left\{ \sum_{k\in \Z}  \left\|2^{k\alpha }\sum_{Q\in Q_d^k}|\lambda_Q|\chi_Q(x)\right\|_p^{q\over p} \right\}^{1\over q} \sim\left\{\sum_{k\in \Z}\left\|2^{k\alpha }\sum_{Q\in Q_d^k}\left|{\bf D}_k(f)(x_Q)\right|\chi_Q(x)\right\|_p^{q\over p}\right\}^{1\over q}.
		$$
		First, we  prove $$
		\inf_{\{\lambda_Q\}_{k\in \Z,Q\in Q_d^k}}\left\{ \sum_{k\in \Z}  \left\|2^{k\alpha }\sum_{Q\in Q_d^k}|\lambda_Q|\chi_Q(x)\right\|_p^{q\over p} \right\}^{1\over q} ~\leq ~C~\left\{\sum_{k\in \Z}\left\|2^{k\alpha }\sum_{Q\in Q_d^k}\left|{\bf D}_k(f)(x_Q)\right|\chi_Q(x)\right\|_p^{q\over p}\right\}^{1\over q}.
		$$
		By {\bf{Theorem \ref{distribution  converge}}}, there exists a distribution $h$ such that $\|h\|_{\dot{B}_{p,d}^{\alpha,q}}\sim\|f\|_{\dot{B}_{p,d}^{\alpha,q}}$,  and
		$$
		f(x)=\sum_{k\in\Z}\sum_{Q\in Q_d^k} \omega(Q)D_k^{M_0}(x,x_Q){\bf D}_k(h)(x_Q)
		$$
		in the distribution sense. Thus,
		\begin{equation}\label{inf definitiopn}
		\begin{aligned}
			\inf_{\{\lambda_Q\}_{k\in \Z,Q\in Q_d^k}}\left\{ \sum_{k\in \Z}  \left\|2^{k\alpha }\sum_{Q\in Q_d^k}|\lambda_Q|\chi_Q(x)\right\|_p^{q\over p} \right\}^{1\over q}
			&\leq \left\{\sum_{k\in \Z}\left\|2^{k\alpha }\sum_{Q\in Q_d^k}\left|{\bf D}_k(h)(x_Q)\right|\chi_Q(x)\right\|_p^{q\over p}\right\}^{1\over q}
			\\
			&=C\left\{\sum_{k\in \Z}\left\|2^{k\alpha }\sum_{Q\in Q_d^k}\left|{\bf D}_k(f)(x_Q)\right|\chi_Q(x)\right\|_p^{q\over p}\right\}^{1\over q}.
		\end{aligned}
		\end{equation}
		It remains to prove that there exist a constant $C$, such that for each $\{\lambda_Q\}_{k\in \Z,Q\in Q^k}$,
		$$
		\left\{\sum_{k\in \Z}\left\|2^{k\alpha }\sum_{Q\in Q_d^k}\left|{\bf D}_k(f)(x_Q)\right|\chi_Q(x)\right\|_p^{q\over p}\right\}^{1\over q} \leq C\left\{ \sum_{k\in \Z}  \left\|2^{k\alpha }\sum_{Q\in Q_d^k}|\lambda_Q|\chi_Q(x)\right\|_p^{q\over p} \right\}^{1\over q}.
		$$
		Because
		$$f(x)=\sum_{k\in\Z}\sum_{Q\in Q_d^k} \omega(Q)D_k^{M_0}(x,x_Q)\lambda_Q $$
		in the distribution sense and  $D_{k} (x_Q,\cdot)\in L^2\cap\dot{B}^{\alpha',q'}_{p',d}$,  we have
		\begin{equation}
			\begin{aligned}
				&\left\{\sum_{k'\in \Z}\left\|2^{k'\alpha }\sum_{Q'\in Q_d^{k'}}\left|{\bf D}_{k'}(f)(x_{Q})\right|\chi_{Q'}(x)\right\|_p^{q\over p}\right\}^{1\over q}
				\\
				&=\left\{\sum_{k'\in \Z}\left\|2^{k'\alpha }\sum_{Q'\in Q_d^{k'}}\left|\sum_{k\in\Z}\sum_{Q\in Q_d^{k}}\omega(Q)D_{k'}D_{k}^{M_0}(x_{Q'},x_{Q})\lambda_Q\right|\chi_{Q'}(x)\right\|^{q\over p}\right\}^{1\over q}
				\\
				&\leq C\left\{\sum_{k\in\Z}\left\|2^{k\alpha }\sum_{Q\in Q_d^k}\lambda_Q\chi_Q(x) \right\|_p^q \right\}^{\frac{1}{q}},
			\end{aligned}
		\end{equation}
		where the last inequality is by the almost orthogonal estimate of $\left|D_{k'}D^{M_0}_k(x,y)\right| $ and Lemma \ref{iterate}.

\section{Boundedness of  Dunkl-Calder\'{o}n-Zygmund operators}
		\setcounter{equation}{0}	
		Now we discuss the boundedness of Dunkl-Calder\'{o}n-Zygmund operators ${\bf T}$ on Dunkl Besov spaces $\dot{\B}_{p,d}^{\alpha,q}$. In this section, we always require that $|\alpha|<1, \max\left\{{N\over {N+1}},~{N\over{N+1+\alpha}}\right\}<p<\infty$, $0<q<\infty$. We always denote by $K(x,y)$ the kernel of ${\bf T}$. So ${\bf T }f(x)=\int_{\R^n} K(x,y) f(y) d\omega(y)$. 
		\begin{proof}[Proof of Theorem \ref{T bound}.]  Since distribution space $\dot{\B}^{\alpha,q}_{p,d}=\bar{ L^2\cap\dot{B}^{\alpha,q}_{p,d}}$, it suffices to prove $\|{\bf T}f\|_{\dot{B}^{\alpha,q}_{p,d}}\leq C \|f\|_{\dot{B}^{\alpha,q}_{p,d}}$ for $f\in L^2\cap\dot{B}^{\alpha,q}_{p,d}$. Because ${\bf T}$ is $L^2$-bounded, for $f\in L^2\cap\dot{B}^{\alpha,q}_{p,d}$, by {\bf Theorem} \ref{Dunkl reproducing} we have
		\begin{equation} 
			\begin{aligned}
				\left\|{\bf T}f \right\|_{\dot{B}^{\alpha,q}_{p,d}}&=\left\{\sum_{{k'}\in\Z}\left\|2^{{k'}\alpha}\sum_{{Q'}\in Q_d^{k'}}{\bf D}_{k'} \left( {\bf T}\left(\sum_{k\in\Z}\sum_{Q\in Q_d^k}\omega(Q)D_k^{M_0}(\cdot,x_Q)D_k(h)(x_Q)\right)\right)(x_{Q'})\chi_{Q'}(x) \right\|_p^q \right\}^{\frac{1}{q}}.
				\\
				&=\left\{\sum_{k'\in\Z} \left\| \sum_{k\in\Z}    \sum_{Q'\in Q_d^{k'}}  \sum_{Q\in Q_d^k}  {2^{k'\alpha} \omega(Q)D_{k'}   { T}D^{M_0}_k(x_{Q'},x_{Q}){\bf D}_k(h)(x_{Q})  \chi{_{Q'}}(x) }\right\|_{p}^{q}\right\} ^{1\over q},
			\end{aligned}
		\end{equation}
		where
		\begin{equation}
			D_{k'}TD_{k}^{M_0}(x,y)=\iint _{\R^n\times\R^n}D_{k'}(x,u)K(u,v)D_k^{M_0}(v,y)d\omega(u)d\omega(v).
		\end{equation}
		
		We now divide above summation into two parts. Let
		\begin{equation}
			t_{k,k'}=
			\begin{cases}
				1\qquad\textit{if $k>k'$},\\
				0\qquad\textit{if $k\leq k'$},
			\end{cases}
		\end{equation}
		and $t'_{k,k'}=1-t_{k,k'}$. Let
		\begin{equation}
			I_1=\left\{\sum_{k'\in\Z} \left\| \sum_{k\in\Z}    \sum_{Q'\in Q_d^{k'}}  \sum_{Q\in Q_d^k}  {2^{k'\alpha} \omega(Q)t_{k,k'}D_{k'}   { T}D^{M_0}_k(x_{Q'},x_{Q}){\bf D}_k(h)(x_{Q})  \chi{_{Q'}}(x) }\right\|_{p}^{q}\right\}^{1\over q}
		\end{equation}
		and
		\begin{equation}
			I_2=\left\{\sum_{k'\in\Z} \left\| \sum_{k\in\Z}    \sum_{Q'\in Q_d^{k'}}  \sum_{Q\in Q_d^k}  {2^{k'\alpha} \omega(Q)t'_{k,k'}D_{k'}   { T}D^{M_0}_k(x_{Q'},x_{Q}){\bf D}_k(h)(x_{Q})  \chi{_{Q'}}(x) }\right\|_{p}^{q}\right\}^{1\over q}.
		\end{equation}
		Then by Minkowski inequality, $\left\|{\bf T}f \right\|_{\dot{B}^{\alpha,q}_{p,d}}\leq C(I_1+I_2)$. Let us denote \begin{equation}\label{Skk'}
			S_{k,k'}(x,y):=\left|D_{k'}{ T}D^{M_0}_k(x,y)\right|.
		\end{equation}

		We first prove $(i)$. Write 
		\begin{equation}
			\widetilde{D}_k(x,y):={T}D^{M_0}_k(x,y)=\int_{\R^n} K(x,u)D_k^{M_0}(u,y)d\omega(u).
		\end{equation}
		Because ${\bf T}1=0$, from  \cite[{\bf Theorem $1.5$}]{Hp}, for any $\varepsilon<\varepsilon_0$, where $\var_0$ is the regularity exponent of ${\bf T}$, we have   \begin{equation}\label{weak1'}
			\left|	\widetilde{D}_k(x,y)\right|~\leq~ C_{\varepsilon}\frac{1}{V(x,y,2^{-k}+d(x,y))}\left(\frac{2^{-k}}{2^{-k}+d(x,y)}\right)^{\varepsilon};
		\end{equation} 
		\begin{equation} \label{weak2'}
			\begin{aligned}
				\left|	\widetilde{D}_k(x,y)-	\widetilde{D}_k(x',y)\right|~\leq~ C_{\varepsilon}&\left(\frac{\left\|x-x'\right\|}{2^{-k}}\right)^{\varepsilon}\left\{\frac{1}{V(x,y,2^{-k}+d(x,y))}\left(\frac{2^{-k}}{2^{-k}+d(x,y)}\right)^{\varepsilon}\right.
				\\ &\left.+\frac{1}{V(x',y,2^{-k}+d(x',y))}\left(\frac{2^{-k}}{2^{-k}+d(x',y)}\right)^{\varepsilon}\right\};
			\end{aligned}
		\end{equation}
		\begin{equation}\label{weak3'}
			\begin{aligned}
				\left|	\widetilde{D}_k(x,y)-	\widetilde{D}_k(x,y')\right|~\leq~ C_{\varepsilon}&\left(\frac{\left\|y-y'\right\|}{2^{-k}}\right)^{\varepsilon}\left\{\frac{1}{V(x,y,2^{-k}+d(x,y))}\left(\frac{2^{-k}}{2^{-k}+d(x,y)}\right)^{\varepsilon}\right.
				\\ &\left.+\frac{1}{V(x,y',2^{-k}+d(x,y'))}\left(\frac{2^{-k}}{2^{-k}+d(x,y')}\right)^{\varepsilon}\right\};
			\end{aligned}
		\end{equation}
		Because $|\alpha|<\var_0$, we can fix an $\varepsilon<\var_0$ such that $\alpha<\var$, and we rewrite $C$ for the constant. We now estimate $t_{k,k'}S_{k,k'}$ and $I_1$. If $k\leq k'$, then $t_{k,k'}=0$, If $k>k'$, we have
		\begin{equation}\label{DTD}
			\begin{aligned}
				S_{k,k'}(x,y)= \left|\int _{\R^n}D_{k'}(x,u)\widetilde{D}_k(u,y)d\omega(u)\right|,
			\end{aligned}
		\end{equation}
		and 
		\begin{equation}\label{case2}
			\begin{aligned}
				&\left|\int _{\R^n}D_{k'}(x,u)\widetilde{D}_k(u,y)d\omega(u)\right|
				\\
				&\leq~C~\int_{\R^n}\frac{1}{V(x,u,{2^{-k'}}+d(x,u))}\left(\frac{2^{-k'}}{{2^{-k'}}+d(x,u)}\right)^{\var}
				\frac{1}{V(u,y,{2^{-k}}+d(u,y))}\left(\frac{2^{-k}}{{2^{-k}}+d(u,y)}\right)^{\var}d\omega(u)
				\\
				&=C\int_{d(x,u)\geq d(u,y)} \frac{1}{V(x,u,{2^{-k'}}+d(x,u))}\left(\frac{2^{-k'}}{{2^{-k'}}+d(x,u)}\right)^{\var}
				\frac{1}{V(u,y,{2^{-k}}+d(u,y))}\left(\frac{2^{-k}}{{2^{-k}}+d(u,y)}\right)^{\var}d\omega(u)
				\\ 
				&~~~~+\int_{d(x,u)< d(u,y)} \frac{1}{V(x,u,{2^{-k'}}+d(x,u))}\left(\frac{2^{-k'}}{{2^{-k'}}+d(x,u)}\right)^{\var}
				\frac{1}{V(u,y,{2^{-k}}+d(u,y))}\left(\frac{2^{-k}}{{2^{-k}}+d(u,y)}\right)^{\var}d\omega(u)
				\\
				&=I+II.
			\end{aligned}
		\end{equation}
		When $d(x,u)\geq d(u,y)$, we have $d(x,y)\leq d(x,u)+d(u,y)\leq2d(x,u)$. Then by \cite[{\bf Lemma $2.2$}] {Hp}, 
		\begin{equation}\label{case2i}
			\begin{aligned}
				I&\leq  C  \frac{1}{V(x,{2^{-k'}}+d(x,y))}\left(\frac{2^{-k'}}{{2^{-k'}}+d(x,y)}\right)^{\var}
				\int_{\R^n}\frac{1}{V(u,y,{2^{-k}}+d(u,y))}\left(\frac{2^{-k}}{{2^{-k}}+d(u,y)}\right)^{\var}d\omega(u)
				\\
				&\leq C \frac{1}{V(x,y,{2^{-k'}}+d(x,y))}\left(\frac{2^{-k'}}{{2^{-k'}}+d(x,y)}\right)^{\var}.
			\end{aligned}
		\end{equation}
		When $d(x,u)<d(u,y)$, we have $d(x,y)\leq d(x,u)+d(u,y)\leq2d(u,y)$. If $d(x,y)>2^{-k'}$, then by \cite[{\bf Lemma $2.2$}] {Hp}, 
		\begin{equation}\label{case2ii}
			\begin{aligned}
				II&\leq  C \frac{1}{V(x,{2^{-k}}+d(x,y))}\left(\frac{2^{-k}}{{2^{-k}}+d(x,y)}\right)^{\var}
				\int_{\R^n}\frac{1}{V(x,u,{2^{-k'}}+d(x,u))}\left(\frac{2^{-k'}}{{2^{-k'}}+d(x,u)}\right)^{\var}d\omega(u)
				\\
				&\leq C \frac{1}{V(x,y,{2^{-k'}}+d(x,y))}\left(\frac{2^{-k'}}{{2^{-k'}}+d(x,y)}\right)^{\var}.
			\end{aligned}
		\end{equation}
		If $d(x,y)\leq 2^{-k'}$, then 
		\begin{equation}
			\begin{aligned}
				II&\leq  C    \frac{1}{V(x,{2^{-k'}})}\left(\frac{2^{-k'}}{{2^{-k'}}}\right)^{\var}
				\int_{\R^n}\frac{1}{V(x,u,{2^{-k}}+d(x,u))}\left(\frac{2^{-k}}{{2^{-k}}+d(x,u)}\right)^{\var}d\omega(u)
				\\
				&\leq C \frac{1}{V(x,y,{2^{-k'}}+d(x,y))}\left(\frac{2^{-k'}}{{2^{-k'}}+d(x,y)}\right)^{\var}.
			\end{aligned}
		\end{equation}
		In conclusion, if $k>k'$, we have
		\begin{equation}
			S_{k,k'}(x,y)\leq C \frac{1}{V(x,y,{2^{-k'}}+d(x,y))}\left(\frac{2^{-k'}}{{2^{-k'}}+d(x,y)}\right)^{\var}.
		\end{equation}
		In general we have
		\begin{equation}\label{ta}
			2^{(k'-k)\alpha}t_{k,k'}S_{k,k'}(x,y)\leq C2^{-|k'-k|\alpha} \frac{1}{V(x,y,{2^{-k'}}+d(x,y))}\left(\frac{2^{-k'}}{{2^{-k'}}+d(x,y)}\right)^{\var}.
		\end{equation}
		We now estimate $I_1$. By definition,
		\begin{equation}
			I_1=\left\{\sum_{k'\in\Z} \left\| \sum_{k\in\Z}    \sum_{Q'\in Q_d^{k'}}  \sum_{Q\in Q_d^k} \omega(Q){2^{(k'-k)\alpha}t_{k,k'}S_{k,k'}(x_{Q'},x_Q)2^{k\alpha}{\bf D}_k(h)(x_{Q}) \chi{_{Q'}}(x) }\right\|_{p}^{q}\right\}^{1\over q}
		\end{equation}
		Because $\alpha>0$ and $p>{N\over N+\alpha}>{N\over N+\varepsilon}$,  we can apply Lemma \ref{iterate} to obtain that
		\begin{equation} 
			I_1\leq C\left\{\sum_{k\in\Z}\left\|2^{k\alpha}\sum_{Q\in Q_d^k}{\bf D}_k(h)(x_{Q}) \chi_Q(x) \right\|_p^q \right\}^{\frac{1}{q}}=C \left\|f \right\|_{\dot{B}^{\alpha,q}_{p,d}}.
		\end{equation}
		
		It remains to estimate $t'_{k,k'}S_{k,k'}$ and $I_2$. When $k>k'$, $S'_{k,k'}=0$; when $k\leq k'$, by \cite [{\bf Lemma $2.11$}] {Hp}, for any $0<\varepsilon'<\varepsilon$,
		\begin{equation}\label{DD}
			\left|D_{k'}\widetilde{D}_k(x,y)\right|\leq  C_{\var'} 2^{-|k'-k|\var'}  {1\over {V(x,y,{2^{-k}}+d(x,y))}}\left({2^{-k}}\over {2^{-k}+d(x,y)}\right)^{\var'} .
		\end{equation}
		Choose $\var'<\var$ such that $p>{N\over N+\var'}>{N\over N+\var'+\alpha}$ and $|\alpha|<\var'$. From Lemma \ref{iterate},  we obtain that
		\begin{equation}\label{case1}
			I_2\leq~C\left\{\sum_{k\in\Z}\left\|2^{k\alpha}\sum_{Q\in Q_d^k}{\bf D}_k(h)(x_{Q}) \chi_Q(x) \right\|_p^q \right\}^{\frac{1}{q}}=C\left\|f \right\|_{\dot{B}^{\alpha,q}_{p,d}}.
		\end{equation}
		Then the proof of $(i)$ is complete.
		
		To prove $(ii)$, we write 
		\begin{equation}
			\overline{D}_{k'}(x,y)=D_{k'}T(x,y)=\int_{\R^n}D_{k'}(x,u)K(u,y)d\omega(u)
		\end{equation} 
		Then
		\begin{equation}
			S_{k,k'}(x,y)=\left|\int_{\R^n}\overline{D}_{k'}(x,u)D_k^{M_0}(u,y)d\omega(u)\right|.
		\end{equation}
		The argument is similar to $(i)$. By Theorem 1.5 in \cite{Hp}, $\overline{D}_k(x,y)$ also satisfies \eqref{weak1'}-\eqref{weak3'}.
		Then by the same argument of estimating $S_{k,k'}$ for $k>k'$ in $(i)$, one can see that for $k<k'$,
		\begin{equation}
			S_{k,k'}(x,y)\leq C \frac{1}{V(x,y,{2^{-k'}}+d(x,y))}\left(\frac{2^{-k'}}{{2^{-k'}}+d(x,y)}\right)^{\var}.
		\end{equation}
		Hence
		\begin{equation}
			2^{(k'-k)\alpha}t'_{k,k'}S_{k,k'}(x,y)\leq C2^{-|k-k'||\alpha|}\frac{1}{V(x,y,{2^{-k'}}+d(x,y))}\left(\frac{2^{-k'}}{{2^{-k'}}+d(x,y)}\right)^{\var}.
		\end{equation}
		Because
		\begin{equation}
			I_2=\left\{\sum_{k'\in\Z} \left\| \sum_{k\in\Z} \sum_{Q'\in Q_d^{k'}}  \sum_{Q\in Q_d^k} \omega(Q){2^{(k'-k)\alpha}t'_{k,k'}S_{k,k'}(x_{Q'},x_Q)2^{k\alpha}{\bf D}_k(h)(x_{Q}) \chi{_{Q'}}(x) }\right\|_{p}^{q}\right\}^{1\over q},
		\end{equation}
		by Lemma \ref{iterate}, when $p>{N\over  N-\alpha},$ we can obtain that
		\begin{equation}
				I_2\leq C\left\{\sum_{k\in\Z}\left\|2^{k\alpha}\sum_{Q\in Q_d^k}{\bf D}_k(h)(x_{Q}) \chi_Q(x) \right\|_p^q \right\}^{\frac{1}{q}}=C \left\|f \right\|_{\dot{B}^{\alpha,q}_{p,d}}.
		\end{equation}
		Also, by the argument of estimating $S_{k,k'}$ for $k\leq k'$ and $I_2$ in $(i)$, one can obtain that for $k\geq k'$,
		\begin{equation}
			S_{k,k'}(x,y)<C 2^{-|k-k'|\var'}\frac{1}{V(x,y,{2^{-k'}}+d(x,y))}\left(\frac{2^{-k'}}{{2^{-k'}}+d(x,y)}\right)^{\var'}.
		\end{equation}
		for any $0<\var'<\var_0$. Choose $|\alpha|<\var'<\var$ such that $p>{N\over N+\var'+\alpha}>{N\over N+\var'}$. Then we can apply Lemma \ref{iterate} to obtain that
		$$
		I_1\leq C\left\|f \right\|_{\dot{B}^{\alpha,q}_{p,d}}.
		$$
		Then the proof of $(ii)$ is complete.

		For $(iii)$, by  \cite [{\bf Theorem $1.5$}]{Hp}, we have $\widetilde{D}_k(x,y)$ (or $\overline{D}_{k'}(x,y)$) satisfies \eqref{size}-\eqref{cancelation} for any $\varepsilon<\varepsilon_0$. Then by  \cite[{\bf Lemma 2.11}]{Hp}, for each $0<\var<\var_0$, we have
		$$
		\left|D_{k'}TD^{M}_k(x,y)\right|=\left|D_{k'}\widetilde{D}_{k}(x,y)\right|\leq C_{\var} 2^{-|k'-k|\var}  {1\over {V(x,y,{2^{-k}}+d(x,y))}}\left({2^{-k}}\over {2^{-k}+d(x,y)}\right)^{\var}.
		$$
		Choose $|\alpha|<\var<\var_0$ such that $p>\max\left\{{N\over N+\var},{N\over N+\var+\alpha}\right\}$. Then by Lemma \ref{iterate}, we have
		\begin{equation}
			\begin{aligned}
				\|Tf\|_{\dot{B}_p^{\alpha,q}}&=\left\{\sum_{k'\in\Z} \left\| \sum_{k\in\Z}    \sum_{Q'\in Q^{k'}_{cw}}  \sum_{Q\in Q^k_d}  {2^{k'\alpha} \omega(Q)D_{k'}\widetilde{D}_k(x_{Q'},x_Q){\bf D}_k(h)(x_{Q})  \chi{_{Q'}}(x) }\right\|_{p}^{q}\right\} ^{1\over q}\\
				&\leq C\left\{\sum_{k'\in\Z} \left\| \sum_{Q\in Q^k_d}  {2^{k'\alpha} \omega(Q){\bf D}_k(h)(x_{Q})  \chi{_{Q'}}(x) }\right\|_{p}^{q}\right\}^{1\over q}= C\|f\|_{\dot{B}_{p,d}^{\alpha,q}}.
			\end{aligned}
		\end{equation}
		The proof of $(iii)$ is complete.
	\end{proof}
\section{Declarations}

{\bf Ethical Approval}~~  Not applicable.

{\bf Funding}~~ Not applicable.

{\bf Availability of data of materials} ~~\textbf{}Not applicable.

\end{document}